\documentclass[a4paper, 11pt]{amsart}

\usepackage{amsfonts}
\usepackage{amsthm}
\usepackage{amsmath, amssymb}
\usepackage{cite}
\usepackage{enumerate}
\usepackage{graphicx}
\usepackage{listings}
\usepackage{color}
\usepackage{listings}
\usepackage{float}
\usepackage{multicol}
\usepackage{lscape}
\usepackage{rotating}
\usepackage{setspace}
\usepackage{caption}
\usepackage{array}
\usepackage[update,prepend]{epstopdf}
\usepackage[pdftex,colorlinks,linkcolor=blue,citecolor=blue]{hyperref}

\usepackage{pinlabel} 

\usepackage[top=2.5cm, bottom=2.5cm, left=2.5cm, right=2.5cm]{geometry}
\newtheorem{thm}{Theorem}[section]
\newtheorem{prop}[thm]{Proposition}
\newtheorem{corol}[thm]{Corollary}
\newtheorem{lemma}[thm]{Lemma}
\theoremstyle{definition}
\newtheorem{mydef}[thm]{Definition}

\newtheorem{remark}[thm]{Remark}

\newcolumntype{L}{>{\centering\arraybackslash}m{3cm}}

\newcommand{\Z}{\mathbb{Z}}

\newcommand{\R}{\mathbb{R}}
\newcommand{\C}{\mathbb{C}}
\newcommand{\Q}{\mathbb{Q}}

\newcommand{\specialcell}[2][c]{%
  \begin{tabular}[#1]{@{}c@{}}#2\end{tabular}}

\newsavebox{\LRmat} 
\savebox{\LRmat}{$\left(\begin{smallmatrix}2&1\\1&1\end{smallmatrix}\right)$}

\begin{document}

\title{Non-geometric veering triangulations}

\author{Craig D. Hodgson}
\address{Department of Mathematics and Statistics \\
         The University of Melbourne \\
         Parkville, VIC, 3010, AUSTRALIA}
\email{craigdh@unimelb.edu.au}
\author{Ahmad Issa}
\curraddr{Department of Mathematics\\
The University of Texas at Austin\\
Austin, TX 78712-1202, USA}

\email{aissa@math.utexas.edu}
\author{Henry Segerman}
\curraddr{Department of Mathematics\\
Oklahoma State University\\
Stillwater, OK 74078, USA }

\email{segerman@math.okstate.edu}

\thanks{C.D.H, H.S were supported by the Australian Research Council grant DP1095760. A.I was supported by
a Master of Science National Scholarship and an Australian Postgraduate Award.}

\begin{abstract}
Recently, Ian Agol introduced a class of ``veering'' ideal triangulations 
for mapping tori of pseudo-Anosov homeomorphisms of surfaces punctured along the singular points.  These triangulations have very special combinatorial properties, and Agol asked if these are ``geometric'',  i.e. realised  in the complete hyperbolic metric with all tetrahedra positively oriented. This paper describes a computer program Veering, building on the program Trains by Toby Hall, for generating these triangulations starting from a description of the homeomorphism as a product of Dehn twists.  Using this we obtain the first examples of non-geometric veering triangulations; the smallest example we have found is a triangulation with 13 tetrahedra.
\end{abstract}

\maketitle

\section{Introduction}

The technique of decomposing finite-volume cusped hyperbolic 3-manifolds into ideal hyperbolic tetrahedra 
was introduced by Thurston \cite{ThNotes} and has proved very useful for understanding these 3-manifolds. 
For example, such triangulations provide an effective means of calculating geometric invariants and 
for computing deformations of hyperbolic structures, as used in the computer programs SnapPea \cite{SnapPea}, Snap  \cite{Snap} and SnapPy \cite{SnapPy}. 
A (topological) {\em ideal triangulation} of a 3-manifold $M$ is a decomposition of $M$ into ideal tetrahedra, 
that is, 3-simplices with their vertices removed, such that their faces are affinely glued in pairs.

A {\em geometric} ideal triangulation is an ideal triangulation where each tetrahedron can be assigned the shape of a positive volume
ideal hyperbolic tetrahedron, such that the tetrahedra glue together coherently to define a global hyperbolic structure which agrees with the complete hyperbolic structure on $M$. A natural question to ask is whether every cusped hyperbolic 3-manifold admits a geometric ideal triangulation. This question remains unanswered.

In this paper we focus on ideal triangulations of fibred cusped hyperbolic 3-manifolds, that is, cusped hyperbolic 3-manifolds of the form $M_\varphi = S \times [0,1]/ (x,0) \sim (\varphi(x), 1)$, where $S$ is a punctured surface and $\varphi : S \rightarrow S$ is a pseudo-Anosov homeomorphism. For each such $\varphi$, Agol gives a construction of an ideal triangulation of $M_{\varphi^{\circ}}$, where $\varphi^{\circ}$ is the restriction of $\varphi$ to the surface obtained by puncturing $S$ at the singular points of the invariant foliations (see \cite{MR2866919}.)

This triangulation is a layered triangulation of $M_{\varphi^\circ}$ and is canonical in the sense that it is uniquely determined by $\varphi$. Furthermore, the triangulation satisfies a combinatorial condition called \emph{veering} (see Section \ref{sec:defns} for the definition). In fact, this triangulation is uniquely characterised by the veering condition, that is, every veering ideal triangulation of $M_{\varphi^{\circ}}$ which is layered with respect to $\varphi^\circ$ is isomorphic to the veering triangulation produced by Agol's construction (see \cite[Proposition 4.2]{MR2866919}). Agol poses the question:

\begin{center}{\bf Question: } \emph{Are the veering triangulations coming from this construction geometric?}\end{center}

In \cite{MR2860987} and \cite{FGAng}, it is shown that veering triangulations admit strict angle structures, which is a necessary condition for an ideal triangulation to be geometric. It can be checked that the well studied monodromy (or Floyd-Hatcher) triangulations are veering, so that they correspond to the triangulations produced by Agol's construction in the case that $S$ is a once-punctured torus. These triangulations are known to be geometric (see \cite{MR2255497} or \cite{MR1988201}.)  Many other examples of geometric veering triangulations were studied in \cite{MR2860987}.

In this paper we produce the first examples of {\em non-geometric} veering triangulations.
Currently, the smallest such example (in terms of the number of tetrahedra) known, 
described in Section \ref{sec:nongeo}, has 13 tetrahedra. It seems unlikely that a counterexample would have been found without a computer search, and it is still something of a mystery why veering triangulations are so frequently geometric.

In Section \ref{chapter:implementation}, we describe a computer program that we implemented, which, given a pseudo-Anosov homeomorphism of an oriented surface (of genus $g \ge 0$ with $p > 0$ punctures) described as a composition of Dehn twists and half twists about adjacent punctures (see Figure \ref{fig:twist_curves}), automates Agol's construction, producing veering triangulation files which can be readily input into the computer program SnapPy. In Section \ref{sec:example} we apply the algorithm given in Section \ref{chapter:implementation} on an example.

We find examples of veering triangulations which SnapPy reports are non-geometric, see Section \ref{chapter:results} and tables of data given in Appendix \ref{app:tables}. In Section \ref{sec:nongeo}, we outline how we rigorously verified that one 13 tetrahedron veering triangulation is not geometric.

In \cite{MR2866919}, Agol briefly mentions how periodic splitting sequences of train tracks give rise to conjugacy invariants which solve the restricted conjugacy problem for pseudo-Anosov mapping classes. As part of our computer program we implemented an algorithm which uses the periodic splitting sequences of train tracks to determine whether or not two pseudo-Anosov mapping classes are conjugate in the mapping class group. This is described in Section \ref{sec:conjtest}.

The program Veering \cite{Veering} described in this paper and tables of results are freely available at \url{http://www.ms.unimelb.edu.au/~veering/}.

We thank Toby Hall for his helpful comments regarding some technical details of the computer program Trains \cite{Trains}, on which our program relies. This paper is primarily based on work done as part of a Master's thesis \cite{Issa} by the second author.

\section{Definitions and background}\label{sec:defns}

Let $S$ be an orientable surface of genus $g$ with $p$ punctures and $\chi(S) < 0$. A \emph{train track} $\tau \subset S$ is a finite $1$-complex with $C^1$ embedded edges, each vertex of which is locally modeled on a switch (see Figure \ref{fig:tt_switch_condition}) so that each vertex has a well defined $1$-dimensional tangent space and which satisfies the following geometry condition \cite{PENNER}. If $R$ is a complementary region of $\tau$ in $S$ then the double of $R$ along $\partial R$ with non-smooth points removed has negative Euler characteristic. Edges of a train track are called \emph{branches}, and vertices are called \emph{switches}. 
A \emph{measured train track} is a train track $\tau$ together with a transverse measure $\mu$, which is a function assigning a positive weight to each edge of $\tau$ such that the switch condition holds, that is, at each switch the sum of the weights of edges on each side of the tangent space are equal (see Figure \ref{fig:tt_switch_condition}.) We sometimes refer to $\tau$ as a measured train track when the measure is understood from context. We also occasionally drop the adjective measured, when it is clear that the train track is measured. We will denote by $k \tau$ the measured train track with the same underlying train track $\tau$ but with measure scaled by $k \in \R_{>0}$. A train track is \emph{trivalent} if every switch has degree $3$. If $\tau$ is a trivalent train track and $e$ is an edge of $\tau$ then it has two ends. An end of $e$ is \emph{large} if it comes into a switch $s$ on the side of the tangent space at $s$ opposite the side with two incident branches, otherwise it is \emph{small}. The branch $e$ is \emph{large} if both of its ends are large, and similarly it is \emph{small} if both of its ends are small.

\begin{figure}[H]
\centerline{\includegraphics[width=200pt]{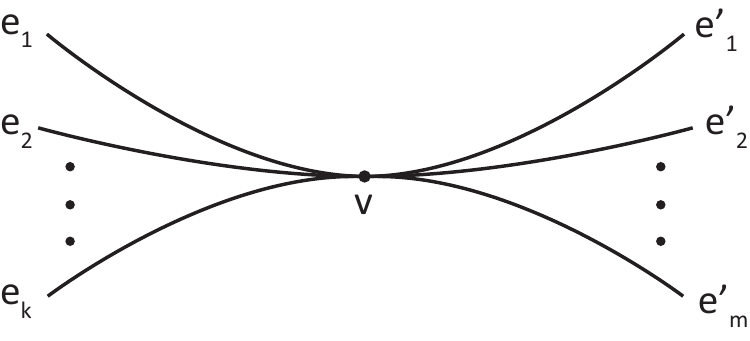}}
\caption{Model of train track switch, where $k,m > 0$. Switch condition: the sum of the weights of edges to the left and to the right of $v$ are equal.}
\label{fig:tt_switch_condition}
\end{figure}

Let $\varphi : S \rightarrow S$ be a homeomorphism. By the Nielsen-Thurston classification \cite{FLP}, $\varphi$ is isotopic to a homeomorphism $\varphi'$ which is either periodic, reducible (setwise fixes a non-empty union of finitely many disjoint essential simple closed curves) or pseudo-Anosov. In the pseudo-Anosov case there exist singular transverse measured foliations $\mathcal{F}^{\pm}$ on $S$ such that $\varphi'(\mathcal{F}^+) = \lambda \mathcal{F}^+$ and $\varphi'(\mathcal{F}^-) = \lambda^{-1} \mathcal{F}^-$, where $\lambda > 1$ is the \emph{dilatation} of $\varphi'$. Away from a finite minimal set of points $P \subset S$, each foliation $\mathcal{F}^{\pm}$ gives a decomposition of $S\backslash P$ into a disjoint union of curves, called \emph{leaves}. A finite line segment lying on a leaf of $\mathcal{F}^-$ can be thought of as scaled by $\varphi'$ by a factor of $\lambda$. The set of points $P$ together with the punctures of $S$ are called \emph{singular points} of $\mathcal{F}^\pm$. Fix a complete finite area hyperbolic metric on $S$. By removing the singular leaves of $\mathcal{F}^{-}$ (resp. $\mathcal{F}^{+}$), i.e. leaves which have an endpoint at a singular point, then isotoping each of the remaining leaves of the foliation to complete geodesic representatives, and finally taking the closure of the resulting subset of $S$, one obtains a geodesic lamination $\mathcal{L}^s$ (resp. $\mathcal{L}^u$) \cite[Construction 1.68]{MR2327361}. The geodesic lamination $\mathcal{L}^s$ (resp. $\mathcal{L}^u$) also inherits a transverse measure from $\mathcal{F}^{+}$ (resp. $\mathcal{F}^{-}$) and is called the \emph{stable} (resp. \emph{unstable}) measured geodesic laminations for $\varphi'$, see \cite{CAS} and \cite{CBNotes} for an alternative approach.

\begin{mydef} Let $(\tau, \mu)$ be a measured train track on a surface $S$, and let $e$ be a large branch with neighbouring edges labelled as in the left of Figure \ref{fig:tt_split_move}.  A \emph{split} at $e$ is a move producing the train track $(\tau', \mu')$ obtained from $(\tau, \mu)$ by splitting $e$ and inserting a new edge $e'$ in one of two possible ways depending on the weights of the neighbouring edges, see Figure \ref{fig:tt_split_move}. We use the notation $(\tau, \mu) \rightharpoonup_e (\tau', \mu')$ to denote that $(\tau', \mu')$ is obtained from $(\tau, \mu)$ by splitting $e$.
{\bf Note:}
	\begin{enumerate}
		\item The weights of all other edges are kept the same. 
		\item The train track $(\tau', \mu')$ is only well-defined up to isotopy. 
		\item We only define a split for $\mbox{max}(a, d) \neq \mbox{max}(b, c)$.
		\item The inverse move which produces $(\tau, \mu)$ from $(\tau', \mu')$ is called a \emph{fold} at $e'$.
	\end{enumerate}
\end{mydef}

\begin{mydef} Let $(\tau, \mu)$ be a measured train track on a surface $S$. A \emph{maximal split} of $(\tau, \mu)$ is a move which produces the train track $(\tau', \mu')$ obtained by splitting all of the edges of $\tau$ that have the maximum weight. This is denoted by $(\tau, \mu) \rightharpoonup (\tau', \mu')$. If $$(\tau_0, \mu_0) \rightharpoonup (\tau_1, \mu_1) \rightharpoonup \cdots \rightharpoonup (\tau_n, \mu_n)$$
is a sequence of $n \in \Z_{>0}$ maximal splits then we write $(\tau_0, \mu_0) \rightharpoonup^{n} (\tau_n, \mu_n)$. Note that in general a maximal split may split more than one edge.
\end{mydef}

\begin{figure}[H]
\centerline{\includegraphics[height=230pt]{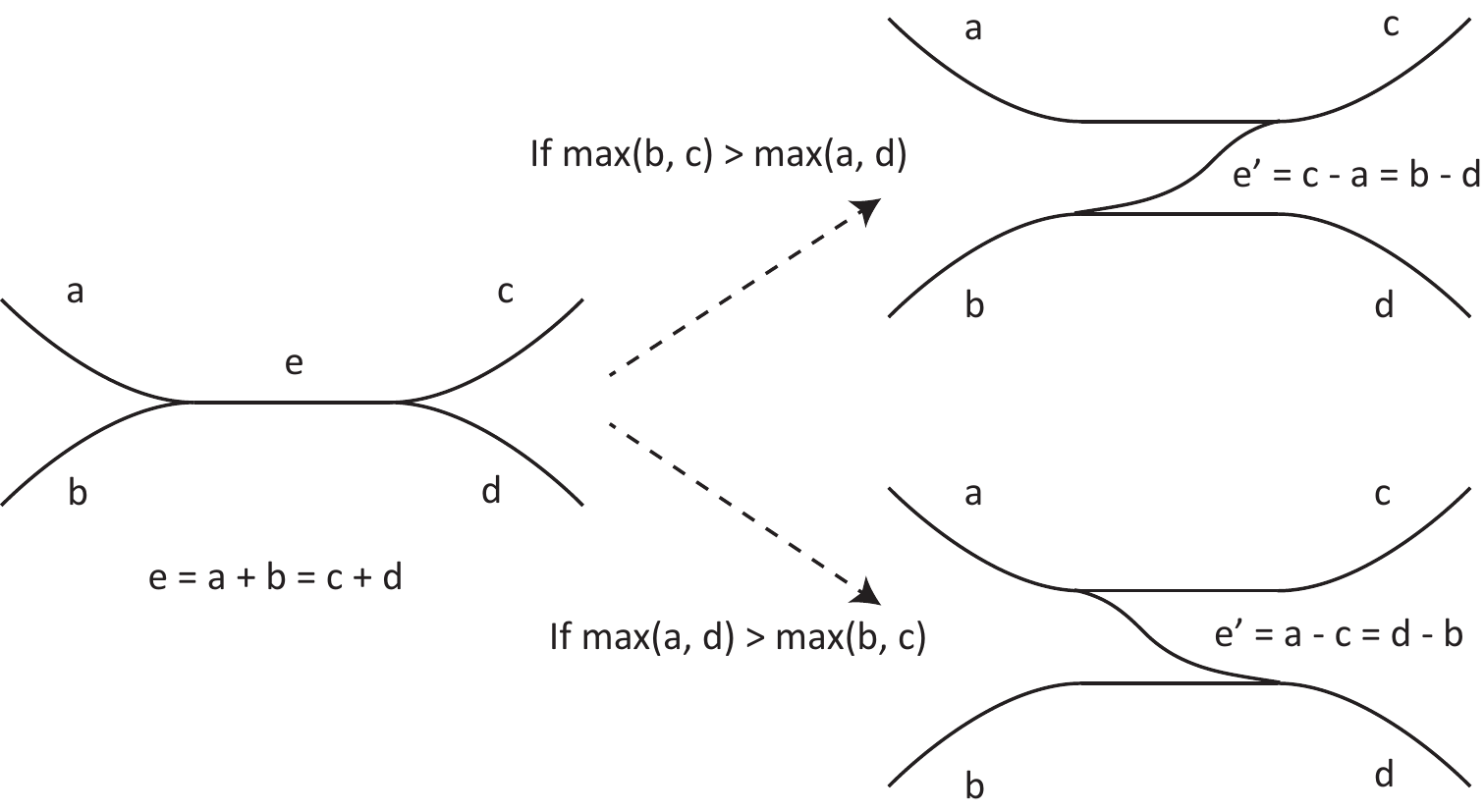}}
\caption{The two possibilities for a split of the large branch $e$. We use the label of an edge to represent its weight.}
\label{fig:tt_split_move}
\end{figure}

\begin{mydef} Let $S$ be a punctured surface. Let $(\mathcal{L}, \lambda)$ be a measured geodesic lamination on $S$, and let $(\tau, \mu)$ be a measured train track on $S$. The measured lamination $\mathcal{L}$ is \emph{suited to} $\tau$ (we also say $\tau$ is suited to $\mathcal{L}$) if the following conditions are satisfied:
\begin{enumerate}[(i)]
	\item There exists a differentiable map $f : S \rightarrow S$ homotopic to the identity such that $f(\mathcal{L}) = \tau$.
	\item $f$ is non-singular on the tangent space to leaves of $\mathcal{L}$, that is, if $v \neq 0$ is a vector tangent to a leaf of $\mathcal{L}$ then $df(v)$ is a non-zero tangent vector to $\tau$.
	\item The map $f$ respects the transverse measures, that is, if $p$ is a point in the interior of an edge $e$ of $\tau$ then $\lambda(f^{-1}(p)) = \mu(e)$.
\end{enumerate}
\end{mydef}

\begin{thm}[Theorem 3.5 of \cite{MR2866919}]\label{thm:splitting} Let $\varphi : S \rightarrow S$ be a pseudo-Anosov homeomorphism with dilatation $\lambda$ and stable measured geodesic lamination $(\mathcal{L}^s, \delta)$. If $(\mathcal{L}^s, \delta)$ is suited to the measured train track $(\tau, \mu)$ then there exist $n,m \in \Z_{>0}$ such that
$$(\tau, \mu) \rightharpoonup^n (\tau_n, \mu_n) \rightharpoonup^m (\tau_{n+m}, \mu_{n+m}),$$
and $\tau_{n+m} = \varphi(\tau_n)$ and $\mu_{n+m} = \lambda^{-1} \varphi(\mu_n)$, where if $e$ is an edge of $\tau_{n+m}$ then $\varphi(\mu_n)(e) := \mu_n(\varphi^{-1}(e))$.
\end{thm}
We call the sequence 
$$(\tau_n, \mu_n) \rightharpoonup \cdots \rightharpoonup (\tau_{n+m}, \mu_{n+m}) \rightharpoonup \cdots, $$
of train tracks in Theorem \ref{thm:splitting} a \emph{periodic splitting sequence}, with respect to $\varphi$, as it is periodic modulo the action of the monodromy and scaling by the dilatation.

We briefly describe Agol's construction of the layered veering triangulation; see also Step 4 of Section \ref{sec:outline} and \cite{MR2866919} for details. With the notation as in Theorem \ref{thm:splitting}, let $\varphi^\circ : S^\circ \rightarrow S^\circ$ be the restriction of $\varphi$ to the surface $S^\circ$ obtained by puncturing $S$ at the singular points of the invariant foliations of $\varphi$. Then the complementary regions of each train track in the periodic splitting sequence are homeomorphic to once-punctured disks. There is an ideal triangulation $T_i$ of $S^\circ$ dual to the train track $\tau_i$, so that edges of $T_i$ are in bijection with branches of $\tau_i$. A split of a train track corresponds to a diagonal exchange of the dual triangulation. Hence, a maximal split corresponds to a sequence of diagonal exchanges interpolating between $T_i$ and $T_{i+1}$ (the veering triangulation is independent of the order of the diagonal exchanges.) Starting with the triangulation $T_n$ of $S^\circ$ we attach a tetrahedron for each diagonal exchange interpolating between the two triangulations. Finally, since $\tau_{n+m} = \varphi^\circ(\tau_n)$ (ignoring measures) we have $T_{n+m} = \varphi^\circ(T_n)$, so we can glue $T_n$ to $T_{n+m}$ to construct a veering triangulation of the mapping torus $M_{\varphi^\circ}$.

We now define veering, a combinatorial condition satisfied by the ideal triangulations produced by Agol's construction.
\begin{mydef}[Taut angle structure]\label{taut_angle_structure}
An {\bf angle-taut tetrahedron} is an ideal tetrahedron equipped with an assignment of angles taken from $\{0,\pi\}$ to its edges so that two opposite edges are assigned $\pi$ and the other four are assigned $0$. A {\bf taut angle structure} on $M$ is an assignment of angles taken from $\{0,\pi\}$ to the edges of each tetrahedron in $M,$ such that every tetrahedron is angle-taut and the sum of all angles around each edge in $M$ is $2\pi.$
\end{mydef}

\begin{mydef}[Taut structure]\label{taut_structure}
A {\bf taut tetrahedron} is a tetrahedron with a coorientation assigned to each face, such that precisely two faces are cooriented into the tetrahedron, and precisely two are cooriented outwards. Each edge of a taut tetrahedron is assigned an angle of either $\pi$ if the coorientations on the adjacent faces agree, or $0$ if they disagree. See Figure \ref{taut_ideal}(a) for the only possible configuration (up to symmetry). Then $\mathcal{T}$ is a {\bf taut ideal triangulation} of $M$ if there is a coorientation assigned to each ideal triangle, such that every ideal tetrahedron is taut, and the sum of all angles around each edge in $M$ is $2\pi$ (see Figure \ref{taut_ideal}(b)). This will also be called a {\bf taut structure} on $M.$
\end{mydef}

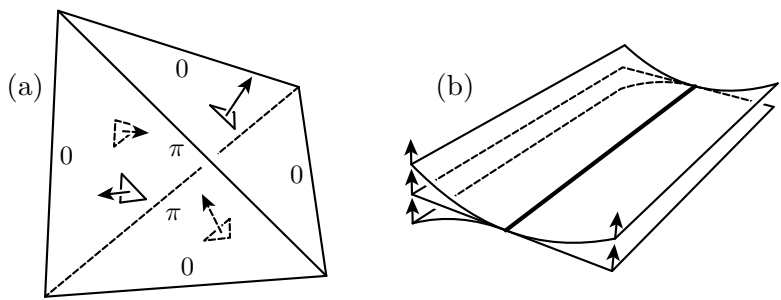
\begin{figure}[ht!]
\labellist
\pinlabel (a) at 60 220
\pinlabel (b) at 465 220
\small
\pinlabel 0 at 213 50
\pinlabel 0 at 208 238
\pinlabel 0 at 316 137
\pinlabel 0 at 99 155
\pinlabel $\pi$ at 200 100
\pinlabel $\pi$ at 203 160
\endlabellist
\centering
\includegraphics[width=0.7\textwidth]{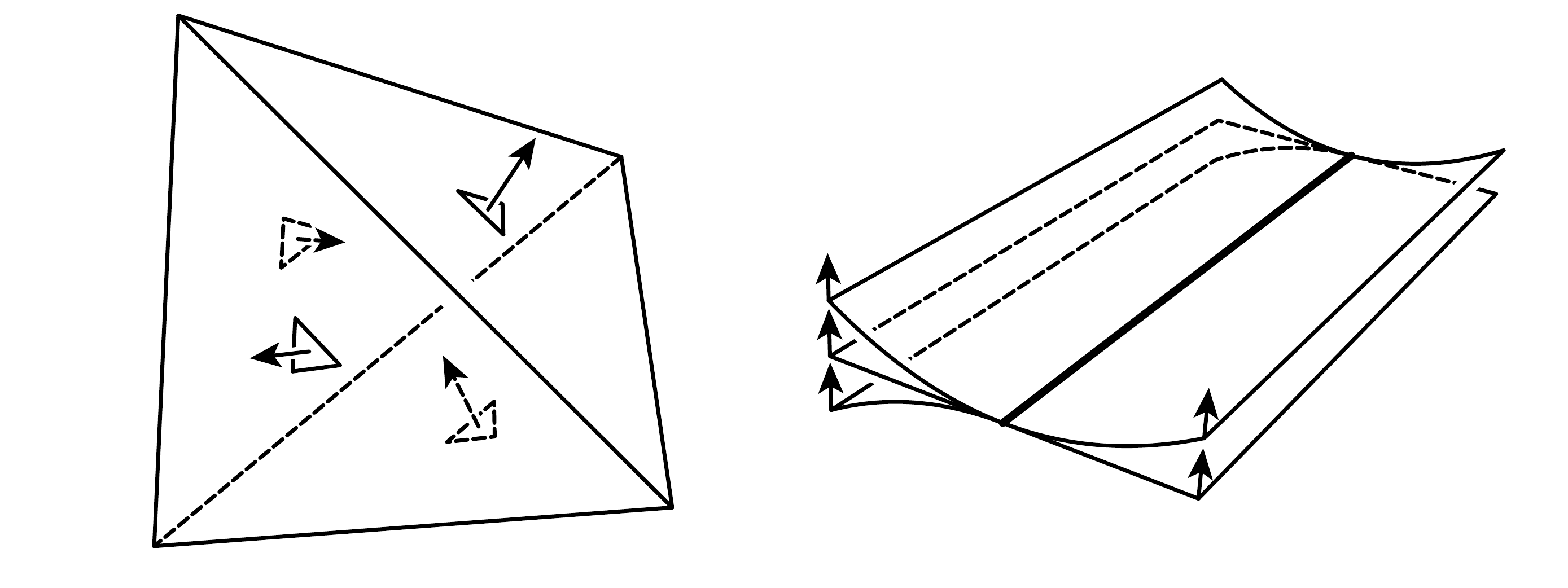}
\caption{Conditions for a taut ideal triangulation.}
\label{taut_ideal}
\end{figure}

A taut ideal triangulation comes with a compatible taut angle structure, but not every taut angle structure arises from a taut structure.

Let $\Delta^3$ be the standard 3--simplex with a chosen orientation. Suppose the edges of $\Delta^3$ are labelled by $e,$ $e'$ and $e'',$ such that opposite edges have the same label and all three labels occur. Then the cyclic order of $e,$ $e'$ and $e''$ viewed from each vertex depends only on the orientation of the 3--simplex, i.e.\thinspace is independent of the choice of vertex. It follows that, up to orientation preserving symmetries, there are two possible labellings, and we fix one of these labellings as shown in Figure~\ref{veering_on_tetrahedron_d}.

\begin{mydef}[Veering triangulation]\label{veering_defn}
A {\bf veering tetrahedron} is an oriented angle-taut tetrahedron, where each edge with angle $0$ is coloured either red or blue (drawn dotted and dashed respectively), such that the cyclic order of the edges at each vertex takes the $\pi$ angle edge to a blue edge to a red edge. This is shown in Figure \ref{veering_on_tetrahedron_d}. We refer to the \underline{r}ed edges as {\bf right-veering} and the b\underline{l}ue edges as {\bf left-veering}. Colours assigned to the $\pi$ angle edges are irrelevant to the definition of a veering tetrahedron. A triangulation $\mathcal{T}$ with a taut angle structure is a {\bf veering triangulation} of $M$ if there is a colour assigned to each edge in the triangulation so that every tetrahedron is veering. 
\end{mydef}

\begin{figure}[htb]
\labellist
\small
\pinlabel $e$ at 52 61
\pinlabel $e'$ at 52 95
\pinlabel $e''$ at 10 50
\endlabellist

\centering
\includegraphics[width=0.2\textwidth]{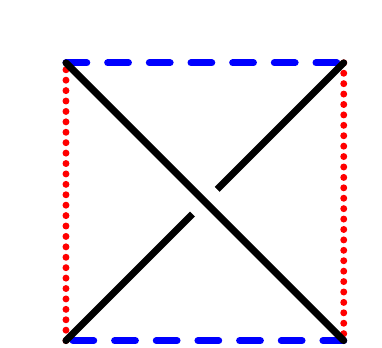}
\caption{The canonical picture of a veering tetrahedron. The $0$ angles are at the four sides of the square, and the $\pi$ angles are the diagonals. We indicate the veering directions on the $0$ angle edges of a tetrahedron by colouring the edges. Note that this picture depends on a choice of orientation for the tetrahedron.}
\label{veering_on_tetrahedron_d}
\end{figure}

The definition of a veering triangulation given above matches the definition given in \cite{MR2860987}. This is slightly more general than the definition given by Agol, which requires that the angle-taut structure on the triangulation is promoted to a taut structure.

\section{Implementation}\label{chapter:implementation}
In this section we discuss a computer program we developed to construct examples of veering triangulations coming from Agol's construction.

Let $S$ be a surface of genus $g \ge 0$ with $p > 0$ punctures, where we label the punctures by integers $1,2,\ldots,p$. Let $\varphi : S \rightarrow S$ be a homeomorphism permuting the punctures, given by a composition $$\varphi = T_{n} \circ \cdots \circ T_{2} \circ T_{1},$$
where each of $T_1, \ldots, T_n$ is either a left or right Dehn twist in one of the curves shown in Figure \ref{fig:twist_curves}, or a half-twist permuting adjacent punctures $i$ and $i+1$, where $i \in \{1,2,\ldots,p-1\}$ (see Figure \ref{fig:half_twist}). The Dehn twists in curves shown in Figure \ref{fig:twist_curves} and half-twists in adjacent punctures generate the mapping class group of $S$ \cite{MR2850125}.

\begin{figure}[H]
\centerline{\includegraphics[width=\textwidth]{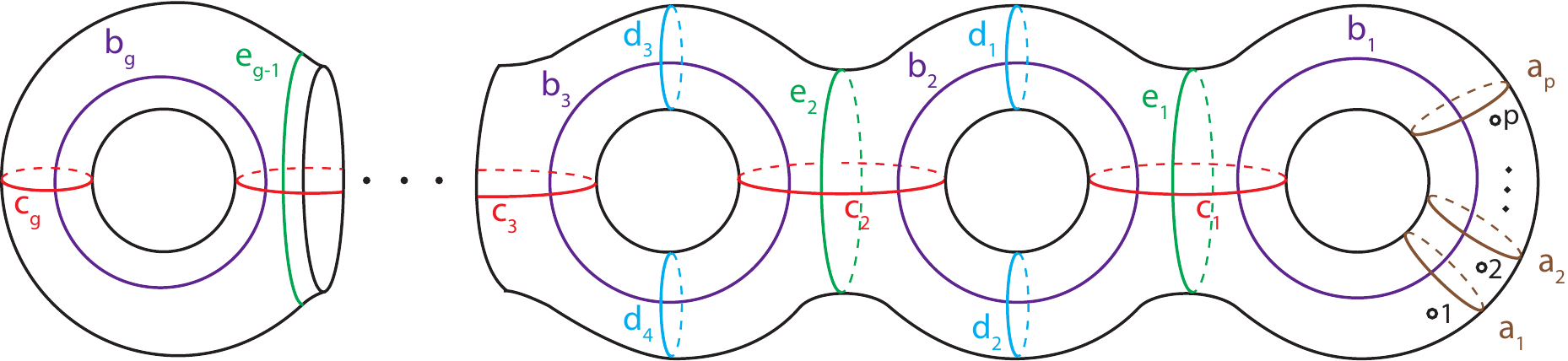}}
\caption{The orientation of $S$ is given by an outward normal vector field and the right hand rule. We consider a positive Dehn twist in a curve shown to be a left twist, that is, we twist to the left as we approach the curve from either side.}
\label{fig:twist_curves}
\end{figure}

\begin{figure}[H]
\centerline{\includegraphics[height=120pt]{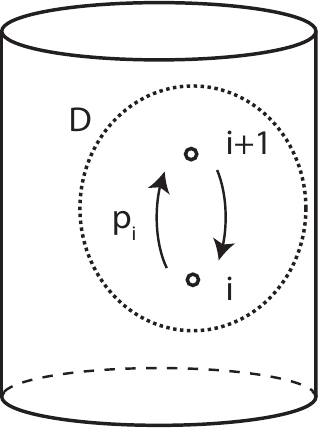}}
\caption{A positive (clockwise) half twist $p_i : S \rightarrow S$ in punctures $i$ and $i+1$, supported in the disk $D$.}
\label{fig:half_twist}
\end{figure}

Let $G \subseteq S$ be a graph homotopy equivalent to $S$. Then $\varphi$ induces a homotopy equivalence $\mathfrak{g} : G \rightarrow G$ and conversely the isotopy class of $\varphi$ is uniquely determined by $\mathfrak{g}$. By a homotopy if necessary, we may assume that $\mathfrak{g}$ is a \emph{graph map}, that is, $\mathfrak{g}$ maps vertices to vertices and each oriented edge to an \emph{edge path}, which is an oriented path $e_1e_2\cdots e_k$, where $e_1,\ldots,e_k$, $k \ge 0$, are oriented edges and the terminal vertex of $e_i$ is equal to the initial vertex of $e_{i+1}$ for $i \in \{1,2,\ldots,k-1\}$. As the graph map $\mathfrak{g}$ is induced by $\varphi$, we say that $\mathfrak{g}$ is a graph map \emph{representing} $\varphi$. We compute such a graph map $\mathfrak{g}$ representing $\varphi$; this is described in more detail in Appendix \ref{sec:comp_graph_map}.

 The graph map produced is input into the computer program Trains \cite{Trains}, written by Toby Hall, which is a software implementation of the Bestvina-Handel algorithm for punctured surfaces (see \cite{BH}.) The Bestvina-Handel algorithm will determine whether or not the homeomorphism represented by the graph map is isotopic to a pseudo-Anosov homeomorphism, and in the case that it is, it will produce a train track $\tau$ suited to the stable geodesic lamination. The program Trains provides the following combinatorial data of the train track:
	\begin{enumerate}
		\item The fat graph structure of $\tau$, i.e. the graph structure together with a cyclic order of incident edges around each vertex respecting the orientation of the surface.
		\item The smoothing at each switch, i.e. the edges on each side of the tangent space.
		\item A cycle of oriented edges of $\tau$ representing a loop on the boundary of each complementary region of $\tau \subset S$ containing a puncture.
	\end{enumerate}

Since each complementary region is a disk or punctured disk, from the above combinatorial data alone we can reconstruct a surface $\Sigma$ diffeomorphic to $S$, with embedded train track $\tau_0$, so that $(\Sigma, \tau_0)$ is diffeomorphic to $(S, \tau)$. The above combinatorial data determines the embedding of the train track in the surface only up to a diffeomorphism of the surface fixing each puncture. That is, if two train tracks $\tau'$ and $\tau''$ on a surface $S$ have isomorphic combinatorial data then there exists a diffeomorphism $h : S \rightarrow S$ pointwise fixing the punctures of $S$, such that $h(\tau') = \tau''$ (see Proposition \ref{prop:ceqdiff} of Section \ref{combeqv}.)

Let $\tau$ be a train track suited to the stable geodesic lamination of $\varphi$, as given by the Bestvina-Handel algorithm. Let $M_{\varphi^{\circ}}$ be the mapping torus of $\varphi^{\circ}$, where $\varphi^{\circ} : S^\circ \rightarrow S^\circ$ is the restriction of $\varphi$ to the surface $S^\circ$ obtained by puncturing $S$ at the singular points of the invariant foliations. As in the Bestvina-Handel algorithm, Trains also outputs a $C^1$ graph map $\mathfrak{g} : \tau \rightarrow \tau$ representing $\varphi^\circ$, with the additional property that the measure on $\tau$ can be computed from $\mathfrak{g}$ (see Step 1 below.) We now outline the steps taken to algorithmically build the veering triangulation of $M_{\varphi^{\circ}}$, given only the combinatorial data of $\tau$ and the map $\mathfrak{g}$, as provided by Trains. Following the outline, we will describe Steps $2$ and $3$ in more detail.

\subsection{Algorithm outline}\label{sec:outline}
		\begin{enumerate}
			\item[Step 1:]{\bf Compute the tranverse measure.} Use the map $\mathfrak{g}$ given by Trains to compute the dilatation $\lambda > 1$ and weights of branches of $\tau$ to some specified arbitrary precision, so that $\tau$ is now a measured train track suited to the stable geodesic lamination of $\varphi^{\circ}$. More precisely, let $e_1, e_2, \ldots, e_m$ be the set of edges of $\tau$, with each edge equipped with an arbitrarily chosen orientation. Let $M$ be the $m \times m$ matrix given by setting $M_{ij}$ to be the number of times either $e_j$ or $\overline{e}_j$ ($e_j$ with its orientation reversed) appears in the edge path $\mathfrak{g}(e_i)$. The dilatation $\lambda$ is the largest real eigenvalue of $M$. Let $v$ be an eigenvector of $M$ with strictly positive entries which spans the one dimensional eigenspace corresponding to $\lambda$. The transverse measure on $\tau$ is given by assigning the $i$th component of $v$ to be the weight of $e_i$, for $i \in \{1,2,\ldots,m\}$. The transverse measure is well defined up to scaling. See \cite{BH} for details.
			\item[Step 2:]{\bf Modify $\tau$ so that it is a trivalent train track (and modify $\mathfrak{g}$ appropriately.)} Agol's construction begins with a trivalent train track suited to the stable geodesic lamination $\mathcal{L}^s$ of $\varphi$, however those produced by the Bestvina-Handel algorithm are generally not trivalent. We modify $\tau$ by applying a `combing' procedure at each switch of $\tau$ with degree greater than $3$ and removing degree $2$ vertices to produce a trivalent measured train track suited to $\mathcal{L}^s$, which we continue to denote by $\tau$.
			\item[Step 3:]{\bf Detect periodicity of splitting sequence of train tracks.} In this step we perform maximal splits starting with $\tau$ to obtain a splitting sequence of train tracks until we determine that the sequence becomes periodic, that is $\varphi^\circ(\lambda^{-1}\tau_i) = \tau_j$ for some $i < j$. Periodicity requires us to check that $\varphi^\circ(\lambda^{-1}\tau_i)$ and $\tau_j$ are identical up to (ambient) isotopy in $S^\circ$. However, since we do not keep track of embeddings of train tracks in $S^\circ$ and therefore cannot directly determine whether two train tracks are isotopic in $S^\circ$, we instead use an indirect approach to detect periodicity. We introduce the notion of combinatorial equivalence of two given measured train tracks on the same surface, which is essentially equivalent to saying that the two tracks have isomorphic combinatorial data. Combinatorial equivalence is a necessary but insufficient condition for two train tracks to be isotopic but can be checked for algorithmically. Moreover, given a combinatorial equivalence between the two train tracks $\lambda^{-1}\tau_i$ and $\tau_j$, there exists a diffeomorphism $h : S^\circ \rightarrow S^\circ$ such that $h(\lambda^{-1}\tau_i) = \tau_j$ which realises the combinatorial equivalence. Then $h$ and $\varphi^\circ$ are isotopic in $S^\circ$ if and only if the induced maps $h_*, \varphi^\circ_* : \pi_1(S^\circ) \rightarrow \pi_1(S^\circ)$ are equal up to an inner automorphism. The map $\varphi^\circ_*$ can be computed from the map $\mathfrak{g}$ given by Trains, and $h_*$ is determined by the combinatorial equivalence alone. If $h$ and $\varphi^\circ$ are determined to be isotopic then $\varphi^\circ(\lambda^{-1}\tau_i) = h(\lambda^{-1}\tau_i) = \tau_j$ so that the sequence is periodic as required.
			\item[Step 4:]{\bf Use the splitting sequence of train tracks to construct the veering triangulation.} We briefly describe the construction of the layered triangulation from the periodic splitting sequence of train tracks and combinatorial equivalence found in Step $3$, see \cite[\S 4]{MR2866919} for details. Assume we have found that $\varphi^\circ(\lambda^{-1}\tau_n) = \tau_{n+m}$, $n,m\in\Z_{>0}$. Recall that we add punctures to $S$ at the singular points of the transverse measured foliations of $\varphi$ to obtain a surface $S^\circ$ so that each complementary region (in $S^\circ$) of a train track in our splitting sequence is homeomorphic to a punctured disk. For a given train track there is a dual ideal triangulation of the surface $S^\circ$ so that edges of ideal triangles are in bijection with branches of the train track. A maximal split $\tau_i \rightharpoonup \tau_{i+1}$ corresponds to a sequence of diagonal exchanges, and for each diagonal exchange an ideal tetrahedron is attached.
			
			Finally, we need to specify how the triangulation $T_n$ of $S^\circ$ dual to $\tau_n$ glues to the triangulation $T_{n+m}$ of $S^\circ$ dual to $\tau_{n+m}$, which is required in order to completely determine the tetrahedron face pairings. The combinatorial equivalence between $\lambda^{-1}\tau_n$ and $\tau_{n+m}$ from Step 3 induces a bijection between the edges of $\tau_n$ and $\tau_{n+m}$, which in turn induces a gluing of the ideal triangles in $T_n$ to those in $T_{n+m}$, as required.
	  \end{enumerate}
			
	\subsection{Step 2}\label{subsec:step2} We now describe Step 2 in more detail. Let $v$ be a switch of $\tau$ with degree greater than $3$. Let $e_1, e_2$ be consecutive edges incident to $v$ on the same side of the tangent space of $\tau$ at $v$, and let $U \subseteq S$ be a small closed disk such that $v \in \mbox{Int}(U)$ (see Figure \ref{fig:combing}.) Identify the arcs $e_1 \cap U$ and $e_2 \cap U$ to produce a train track, obtained from $\tau$ by \emph{combing}. By repeatedly combing we can obtain a trivalent train track, which we will denote by $\tau'$. The new edge is assigned a weight equal to the sum of the weights of edges $e_1$ and $e_2$. By Proposition \ref{combing} below, $\mathcal{L}$ is suited to $\tau'$. (Further discussion of combing can be found on page 40 of \S 1.4 of \cite{PENNER}.) 
	
	Assume first that $\tau'$ is obtained from $\tau$ by combing a single edge as in Figure \ref{fig:combing}. We modify the graph map $\mathfrak{g}$ to obtain a graph map $\tau' \rightarrow \tau'$ which we will continue to denote by $\mathfrak{g}$. This graph map will be used to detect periodicity in the splitting sequence of train tracks in Step 4. Let $E$ be the new edge introduced by combing which we orient away from $v$. Assume that the other edges of $\tau'$ are labelled the same as those of $\tau$ and that $e_1, e_2$ are oriented away from the vertex common with $E$ in $\tau'$ as in Figure \ref{fig:combing}. We think of the old edge $e_i$ as now being $E e_i$, $i \in {1,2}$. Set $\mathfrak{g}(E)$ to be the empty edge path. For each remaining edge $e$, and for each $e_i \in \{e_1,e_2\}$ modify $\mathfrak{g}(e)$ by replacing each occurrence of $e_i$ in $\mathfrak{g}(e)$ with $E e_i$ and each occurrence of $\overline{e}_i$ with $\overline{e}_i \overline{E}$. If $\tau'$ is obtained by combing more than one edge of $\tau$ then each time we comb an edge we update $\mathfrak{g}$ as described. The resulting map $\mathfrak{g}$ is a graph map representing $\varphi^\circ$.
	
	Assume that $v'$ is a degree $2$ vertex of $\tau$ with incident edges $e_3$ and $e_4$, oriented away from $v'$. We declare the vertex to be removed from $\tau$ so that the union of $e_3$ and $e_4$ form a new edge $E = \overline{e_3}e_4$. Then we modify the graph map $\mathfrak{g}$ so that $E$ maps to $\mathfrak{g}(\overline{e_3}e_4)$, and we remove all occurrences of $e_3$ and $\overline{e}_3$ in the image of each edge. The edge $e_3$ is thought of as contracted into a vertex and $e_4$ lengthened to form $E$.
			
			\begin{figure}[H]
			\centerline{\includegraphics[height=135pt]{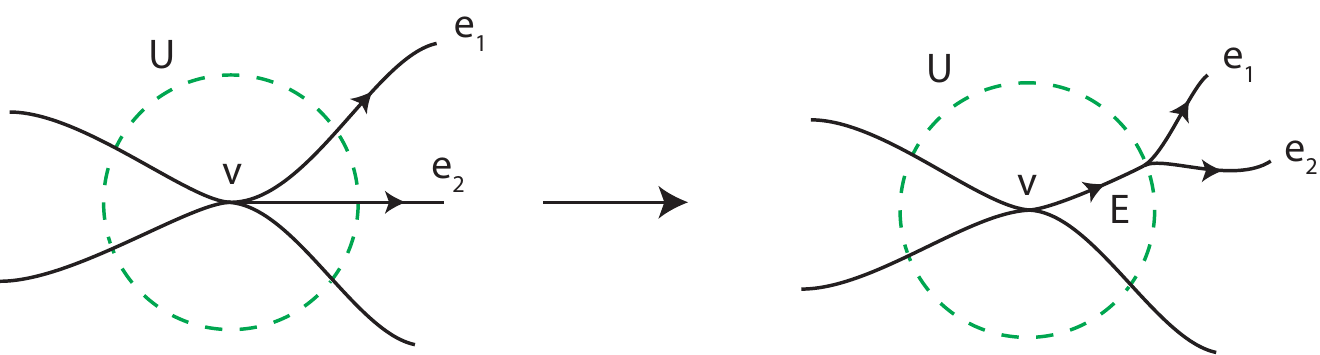}}
			\caption{Combing a train track}
			\label{fig:combing}
			\end{figure}
			
			\begin{prop}\label{combing} Let $\mathcal{L}$ be a lamination suited to a train track $\tau' \subseteq S$. Let $\tau''$ be obtained from $\tau'$ by combing the edges $e_1, e_2$ of $\tau'$ with respect to the closed disk $U \subseteq S$. Then $\mathcal{L}$ is suited to $\tau''$.
			\end{prop}
			
			\begin{proof}
				By definition since $\mathcal{L}$ is suited to $\tau'$ there exists a differentiable map $f : S \rightarrow S$ homotopic to the identity such that $f(\mathcal{L}) = \tau'$ and $f$ is non-singular on the tangent space to leaves of $\mathcal{L}$. Let $h : S \rightarrow S$ be a differentiable map homotopic to the identity which, intuitively, folds together $e_1 \cap U$ and $e_2 \cap U$. Then $h \circ f : S \rightarrow S$ is a differentiable map homotopic to the identity such that $h\circ f(\mathcal{L}) = \tau''$ and is non-singular on the tangent space to leaves of $\mathcal{L}$. The measure is also appropriately preserved by $h \circ f$.
			\end{proof}

			\subsection{Step 3}\label{combeqv} We denote by $\varphi^\circ : S^\circ \rightarrow S^\circ$ the restriction of $\varphi$ to the surface $S^\circ$ obtained by puncturing $S$ at the singular points of the invariant foliations. In this section we outline an effective procedure to detect the periodic splitting sequence associated with $\varphi^\circ$, given only $\tau$ and $\mathfrak{g} : \tau \rightarrow \tau$ as produced by Trains. 
			
			\subsubsection{Combinatorial equivalence of train tracks.}
			Recall that we would like to perform maximal splits beginning with the train track $\tau$, producing a sequence of measured train tracks 
			$$\tau \rightharpoonup \tau_1 \rightharpoonup \cdots \rightharpoonup \tau_n \rightharpoonup \cdots \rightharpoonup \tau_{n+m} \rightharpoonup \cdots, \quad n,m\in\Z_{>0},$$
			 until we find that the sequence becomes periodic, that is, $\lambda^{-1}\varphi^\circ(\tau_n) = \tau_{n+m}$ for some $n,m \in \Z_{>0}$.
			 However, we do not keep track of the embedding of train tracks in the surface $S$ so we can not directly compare train tracks up to isotopy in $S$. Despite this, we can determine if two measured train tracks $\tau', \tau''$ are \emph{combinatorially equivalent}, that is, there is a bijection between their vertices and oriented edges which:
				\begin{enumerate}[(i)]
					\item induces an isomorphism of the combinatorial graph structure of the train tracks,
					\item preserves the clockwise cyclic order of oriented edges around each switch,
					\item preserves the smoothing at each switch (i.e. the list of edges on each side of the tangent space),
					\item maps edges to edges of the same weight, and
					\item preserves complementary regions, that is, if $p$ is a puncture of $S$, then the clockwise sequence of edges which form the boundary of the complementary region of $\tau$ containing $p$ are mapped by the bijection to the clockwise sequence of edges which form the boundary of the complementary region of $\tau'$ containing $p$.
				\end{enumerate}
			See \cite[\S 3.15]{Mosher2} for further discussion on combinatorial equivalence (note however, that we require the additional property that a combinatorial equivalence respects weights of edges.) We write $\tau' \sim \tau''$ to denote combinatorial equivalence. A combinatorial equivalence between two train tracks $\tau'$ and $\tau''$ induces a $C^1$-diffeomorphism $\tau' \rightarrow \tau''$ respecting the bijection of edges and vertices. In fact we have the following proposition.
			
				\begin{prop}\label{prop:ceqdiff} Let $\tau', \tau''$ be combinatorially equivalent measured train tracks on a punctured surface $S$ each of which fill the surface, that is, the complementary regions of each of the train tracks are homeomorphic to disks or once-punctured disks. Then there exists a diffeomorphism $\phi : S \rightarrow S$ such that $\phi(\tau') = \tau''$ which induces the combinatorial equivalence.
				\end{prop}
				See \cite[Proposition 3.15.1]{Mosher2}.

Note that train tracks output by the Bestvina-Handel algorithm always fill the surface, and performing a maximal split on a filling train track results in a filling train track.

\subsubsection{Compute graph maps $\mathfrak{g}_i : \tau_i \rightarrow \tau_i$ representing $\varphi^\circ$, for $i \in \Z_{\ge 0}$.}

Singular points of the invariant foliations of $\varphi$ are in bijection with the complementary regions of $\tau \subset S$. Thus, in $S^\circ$ all complementary regions of $\tau$ (and $\tau_i$ for $i\in\Z_{>0}$) are punctured disks. For $i \in \Z_{>0}$, the train track $\tau_i$ is homotopy equivalent to $S^\circ$ and $\varphi^\circ$ induces a homotopy equivalence $\tau_i \rightarrow \tau_i$. Each time we perform a maximal split we compute such a homotopy equivalence, in fact, a graph map $\mathfrak{g}_i : \tau_i \rightarrow \tau_i$ representing $\varphi^\circ$, which is later used to compute the map $\varphi^\circ_* : \pi_1(S^\circ) \rightarrow \pi_1(S^\circ)$ mentioned in the outline in Section \ref{sec:outline}. This is done inductively, starting with the graph map $\mathfrak{g}_0 := \mathfrak{g}$ computed in Step 3, which was obtained by modifying the graph map produced by the Bestvina-Handel algorithm. Let $i \in \Z_{\ge 0}$ and assume that $\tau_{i+1}$ is obtained from $\tau_i$ by splitting a single large branch $E$, as shown in Figure \ref{fig:tt_split_isotopy}. The edge $e_1$ of $\tau_{i+1}$ can be thought of as represented by $\overline{E}e_1$ in $\tau_{i}$. Similarly, the edge $e_4$ of $\tau_{i+1}$ can be thought of as represented by $Ee_4$ of $\tau_{i}$. All other edges of $\tau_{i+1}$ are thought of as the same as the corresponding edges of $\tau_{i}$. We also think of $e_1$ of $\tau_i$ as being represented by $Ee_1$ of $\tau_{i+1}$ (and analogously for $e_2$.) This gives rise to the following modifications.

First define $\mathfrak{g}'_i : \tau_{i+1} \rightarrow \tau_i$ by setting $\mathfrak{g}'_i(e_1) = \mathfrak{g}_i(\overline{E}e_1)$ and $\mathfrak{g}'_i(e_4) = \mathfrak{g}_i(E e_4)$ (provided that $e_1 \neq \overline{e_4}$, otherwise set $\mathfrak{g}'_i(e_1) = \mathfrak{g}_i(\overline{E}e_1\overline{E})$) and $\mathfrak{g}'_i(e) = \mathfrak{g}_i(e)$ for all other edges $e$ of $\tau_{i+1}$, where we label the oriented edges of $\tau_{i+1}$ the same way as $\tau_i$ for edges not shown in Figure \ref{fig:tt_split_isotopy} and a bar above an edge indicates the edge with its direction reversed. Finally, $\mathfrak{g}_{i+1} : \tau_{i+1} \rightarrow \tau_{i+1}$ is obtained from $\mathfrak{g}'_i$ as follows. If $e$ is an edge of $\tau_{i+1}$ we set $\mathfrak{g}_{i+1}(e)$ to be the edge path given by replacing each edge $e_1$ by $E e_1$, and $e_4$ by $\overline{E}e_4$ in the edge path $\mathfrak{g}'_i(e)$ (and similarly for $\overline{e_1}$ and $\overline{e_4}$.) The case where $E$ is split the other way is dealt with analogously. If $\tau_{i+1}$ is obtained from $\tau_i$ by splitting multiple large branches, then we perform the above procedure for each split in any order.

\begin{remark} In general for an edge $e$ the edge path $\mathfrak{g}_i(e)$, $i\in\Z_{>0}$, may enter and exit a vertex via edges on the same side of the tangent space at a switch.\end{remark}
			
\begin{figure}[H]
\centerline{\includegraphics[width=\textwidth]{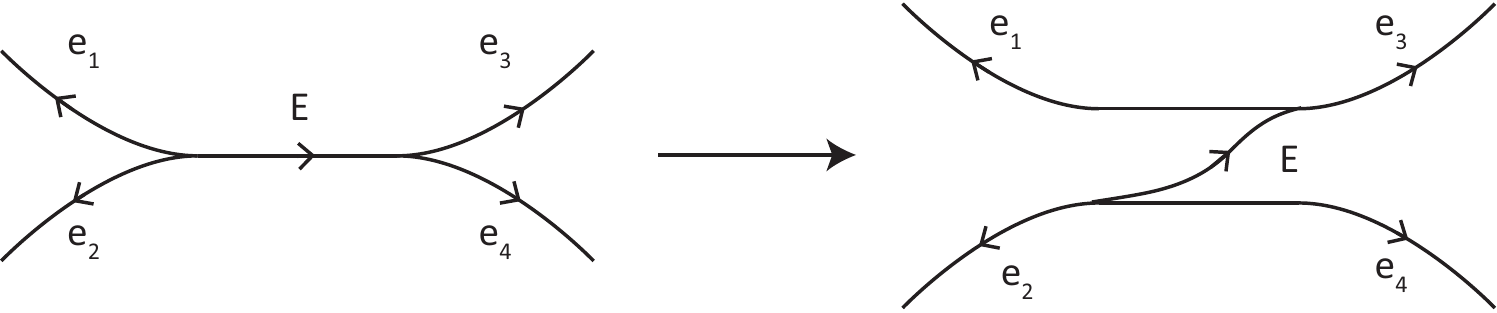}}
\caption{$\tau_i \rightharpoonup \tau_{i+1}$. By convention the new edge is also labelled $E$. The orientations of edges have been chosen arbitrarily.}
\label{fig:tt_split_isotopy}
\end{figure}
			
\subsubsection{Determine when the sequence of train tracks becomes periodic.}
			For each train track $\tau_j$ computed, we enumerate all combinatorial equivalences $\frac{1}{\lambda}\tau_i \sim \tau_j$ with $i < j$. Note that in general two train tracks may be combinatorially equivalent in multiple distinct ways. Suppose that $\frac{1}{\lambda}\tau_i \sim \tau_j$ and fix a combinatorial equivalence. By Proposition \ref{prop:ceqdiff}, the combinatorial equivalence is induced by a homeomorphism $\psi : S^\circ \rightarrow S^\circ$ such that $\frac{1}{\lambda}\psi(\tau_i) = \tau_j$. We determine whether $\psi$ and $\varphi^\circ$ are isotopic as homeomorphisms of $S^\circ$ by determining whether they induce the same outer automorphism of $\pi_1(S^\circ)$.
			
	  For $k \in \Z_{>0}$, there exists a map $h_k : S^\circ \rightarrow S^\circ$ homotopic to the identity, mapping $\tau_{k+1}$ onto $\tau_k$, which folds at the small branches of $\tau_{k+1}$ that arose by the maximal splitting $\tau_k \rightharpoonup \tau_{k+1}$. Then $h_k$ restricts to a map $h_k : \tau_{k+1} \rightarrow \tau_k$. For example, if $\tau_k \rightharpoonup \tau_{k+1}$ is as in Figure \ref{fig:tt_split_isotopy} with $i = k$, then $h_k(e_1) = \overline{E}e_1$, $h_k(e_4) = Ee_4$ and all other edges are mapped to the corresponding edge of $\tau_k$. By composing the restrictions, we obtain a map $f := h_{i} \circ h_{i+1} \circ \cdots \circ h_{j-1} : \tau_{j} \rightarrow \tau_i$, which is induced by a map $S^\circ \rightarrow S^\circ$ homotopic to the identity.
		
	 Let $\mathfrak{g}_i' := f \circ \psi : \tau_i \rightarrow \tau_i$. Now $\mathfrak{g}_i'$ is a graph map representing $\psi$, and $\mathfrak{g}_i$ is a graph map representing $\varphi^\circ$, both in terms of the underlying train track $\tau_i$. Hence it suffices to determine whether they are homotopic graph maps. We do so by checking if they induce the same outer automorphism of $\pi_1(S^\circ) \cong \pi_1(\tau_i)$.
	
	Let $v$ be a vertex of $\tau_i$. By a homotopy if necessary, we can assume that $\mathfrak{g}_i'$ and $\mathfrak{g}_i$ fix $v$. Then $\psi$ and $\varphi$ are isotopic in $S^\circ$ if and only if $(\mathfrak{g}_i')_*$ and $(\mathfrak{g}_i)_*$ are equal as outer automorphisms of the free group $\pi_1(\tau_i, v)$, that is, there exists $\alpha \in \pi_1(\tau_i, v)$ such that $(\mathfrak{g}_i')_*(\gamma) = \alpha \cdot (\mathfrak{g}_i)_*(\gamma) \cdot \alpha^{-1}$ for all $\gamma \in \pi_1(\tau_i, v)$. Noting that $\pi_1(\tau_i, v)$ is a free group on finitely many generators, we can algorithmically check this last condition as follows.
	
	Let $F_n = \langle a_1, a_2, \ldots, a_n \rangle$, be the free group on $n > 1$ generators, and let $f, g : F_n \rightarrow F_n$ be group automorphisms. Our goal is to find $x \in F_n$ such that $f(a_i) = x g(a_i) x^{-1}$ for $i \in \{1,2,\ldots,n\}$, or show that such an $x$ does not exist. If there exists such an $x$, then the cyclically reduced parts of $f(a_i)$ and $g(a_i)$ are equal for $i \in \{1,2,\ldots,n\}$, so we can write $f(a_1) = x_1 g(a_1) x_1^{-1}$ and $f(a_2) = x_2 g(a_2) x_2^{-1}$, for some $x_1, x_2 \in F_n$. Since $f$ is an automorphism we know that $f(a_1),\ldots,f(a_n)$ forms a basis for $F_n$. Hence for $i \in \{1,2,\ldots,n\}$, the stabiliser of $f(a_i)$ under the conjugation action is $\{f(a_i)^k\, |\, k \in \Z\}$. We have $x^{-1} f(a_1) x = x_1^{-1} f(a_1) x_1$, so that $x x_1^{-1} = f(a_1)^k$ for some $k \in \Z$. Similarly, $x = f(a_2)^m x_2$ for some $m \in \Z$. Hence $f(a_2)^{-k} f(a_1)^m = x_2 x_1^{-1}$. A priori, if we did not know whether $f$ and $g$ were equal as outer automorphisms of $F_n$, then we could try to solve this last equation for $k, m \in \Z$. Since $x_2 x_1^{-1}$ as a reduced word in the generators is of finite length, there are only finitely many possible choices of $k, m$ to check, otherwise as a reduced word $f(a_2)^{-k} f(a_1)^m$ would have length exceeding the length of $x_2 x_1^{-1}$. If a solution exists, it is unique since $f(a_1), \ldots, f(a_n)$ forms a basis for $F_n$ so $f(a_2)^{-k} f(a_1)^m$ is distinct for distinct pairs $(k,m)$. Finally, if we are able to find $k, n$ solving the equation, then letting $x = f(a_1)^k x_1$, we can check whether $f(a_i) = x g(a_i) x^{-1}$ for $i = \{1,2,\ldots,n\}$, as required.

\section{Example}\label{sec:example}
Let $T$ be the once-punctured torus given by identifying opposite sides of the square with its vertices removed, as shown in Figure \ref{torus_twist_curves}. Let $f : T \rightarrow T$ be the homeomorphism given by a left Dehn twist in $c_2$, followed by a right Dehn twist in $c_1$, where $c_1$ and $c_2$ are the curves shown in Figure \ref{torus_twist_curves}.
\begin{figure}[H]
\centerline{\includegraphics[height=180pt]{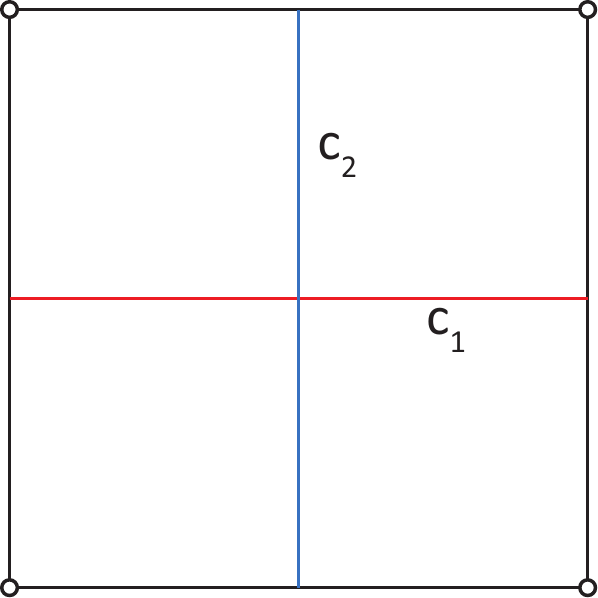}}
\caption{We orient the surface using an anticlockwise rotation direction.}
\label{torus_twist_curves}
\end{figure}

By applying the Bestvina-Handel algorithm we find that $f$ is a pseudo-Anosov homeomorphism with measured train track $\tau$ suited to the stable measured geodesic lamination as shown in Figure \ref{fig:torus_inv_tt}. In brief, we will convert $\tau$ into a trivalent train track. We then split the train track until we obtain the periodic splitting sequence, and construct the ideal triangulation of the mapping torus $M_f$.

\begin{figure}[H]
\centerline{\includegraphics[height=180pt]{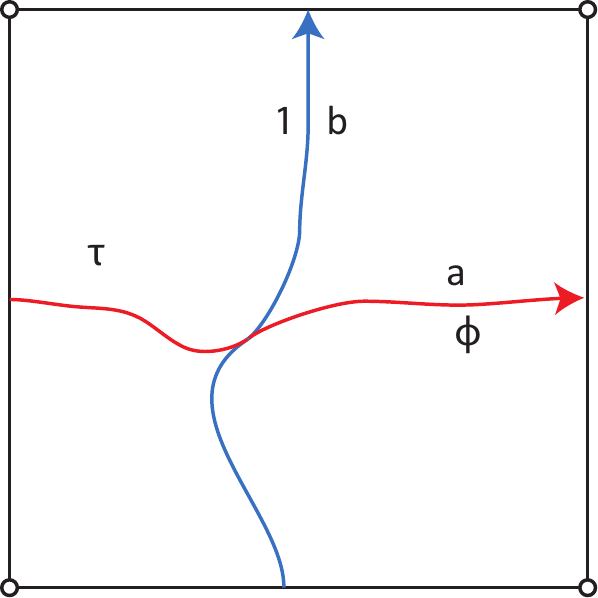}}
\caption{Train track $\tau$, where $\phi = \frac{1}{2}(1+\sqrt{5})$ is the golden ratio.}
\label{fig:torus_inv_tt}
\end{figure}
\vspace{-3mm}

The initial train track map on $\tau$ is given by
\begin{eqnarray*}
\mathfrak{g} : \tau &\rightarrow& \tau \\
a &\mapsto& aba \\
b &\mapsto& ba
\end{eqnarray*}

The train track $\tau$ has one vertex of degree four. As in Section \ref{subsec:step2} we comb edges $a$ and $b$, introducing a new edge $c$ as shown in Figure \ref{fig:torus_trivalent_tt} below. We continue to denote the resulting trivalent train track by $\tau$. The weight of edge $c$ is the sum of the weights of edges $a$ and $b$.

\begin{figure}[H]
\centerline{\includegraphics[height=180pt]{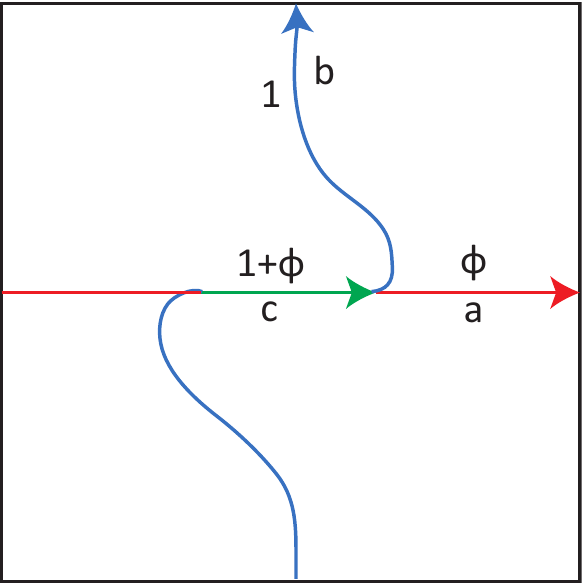}}
\caption{Trivalent weighted train track $\tau$.}
\label{fig:torus_trivalent_tt}
\end{figure}

The resulting map $\mathfrak{g}'$ on $\tau$ is given by
\begin{eqnarray*}
\mathfrak{g}' : \tau &\rightarrow& \tau \\
a &\mapsto& cacbca \\
b &\mapsto& cbca \\
c &\mapsto& \mbox{(empty path)}
\end{eqnarray*}
which is computed from $\mathfrak{g}$ by replacing $a$ by $ca$ and $b$ by $cb$ in the image of each edge, and mapping $c$ to the empty path ($c$ gets mapped into a vertex.) We will drop primes and denote $\mathfrak{g}'$ by $\mathfrak{g}$.

Let $\tau_0 = \tau$, as this will be our initial train track. We compute the image of $\tau_0$ after applying $f$, see Figure \ref{fig:torus_tt_image}. We will first detect the periodic splitting sequence by drawing pictures of train tracks on the surface keeping track of the embedding of the train tracks on the surface. We then execute our algorithm to detect periodicity combinatorially, which does not require train track embedding information, verifying that both methods agree.

\begin{figure}[H]
\centerline{\includegraphics[height=180pt]{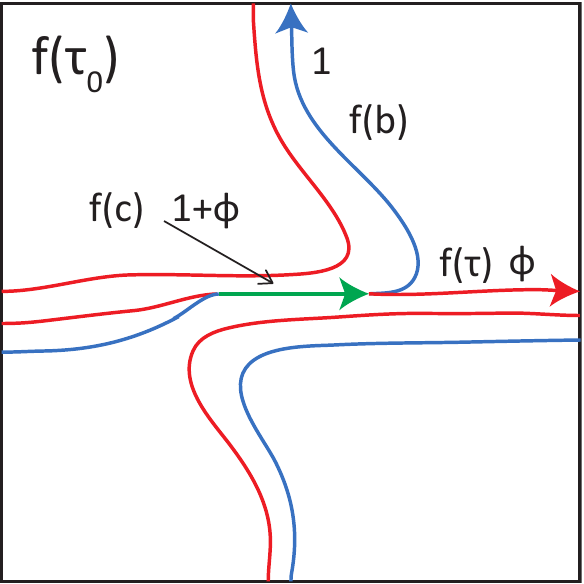}}
\caption{$f(\tau_0)$ up to isotopy.}
\label{fig:torus_tt_image}
\end{figure}

In Figure \ref{fig:torus_tt_split}, we illustrate three maximal splits starting with $\tau_0$.

\begin{figure}[H]
\centerline{\includegraphics[width=380pt]{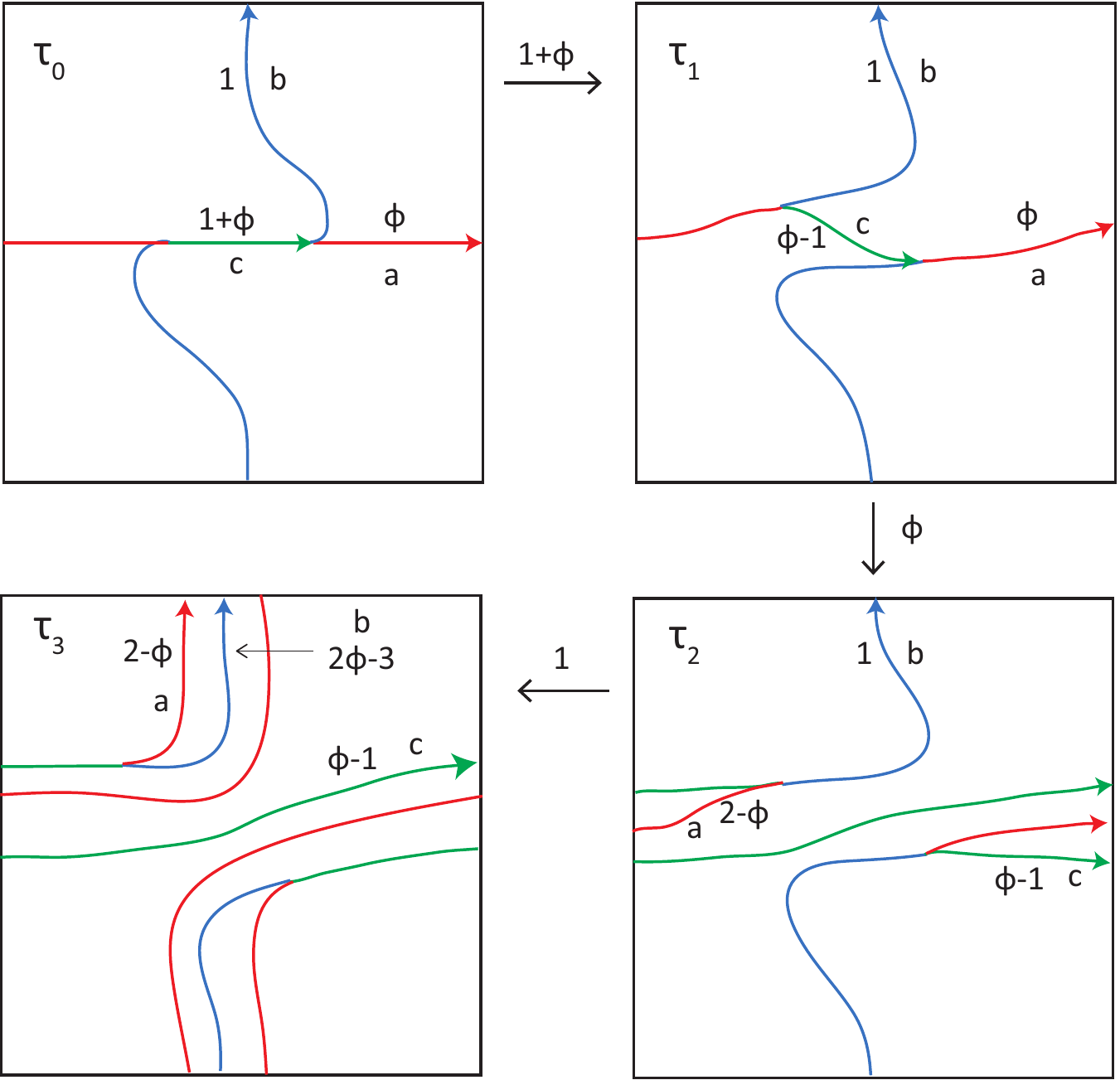}}
\caption{Splitting sequence. We have indicated the weight of the branch undergoing a split.}
\label{fig:torus_tt_split}
\end{figure}

We claim that $f(\lambda^{-1}\tau_1) = \tau_3$, that is, $f(\tau_1)$ and $\tau_3$ are isotopic in such a way that after scaling the weights of $\tau_3$ by the dilatation $\lambda = 1+\phi$, the branches of the train tracks which are identified have equal weights. We compute $f(\tau_1)$ by splitting $f(\tau_0)$ (Figure \ref{fig:torus_tt_image}), this is shown in Figure \ref{fig:torus_tt_isotopy}. We see that $f(\tau_1) = \tau_3$, given by $f(a) = c$, $f(b) = a$ and $f(c) = b$ and that the weights are preserved after scaling, e.g. $\lambda \mu_3(a) = \lambda (2-\phi) = 1$.

\begin{figure}[H]
\centerline{\includegraphics[width=380pt]{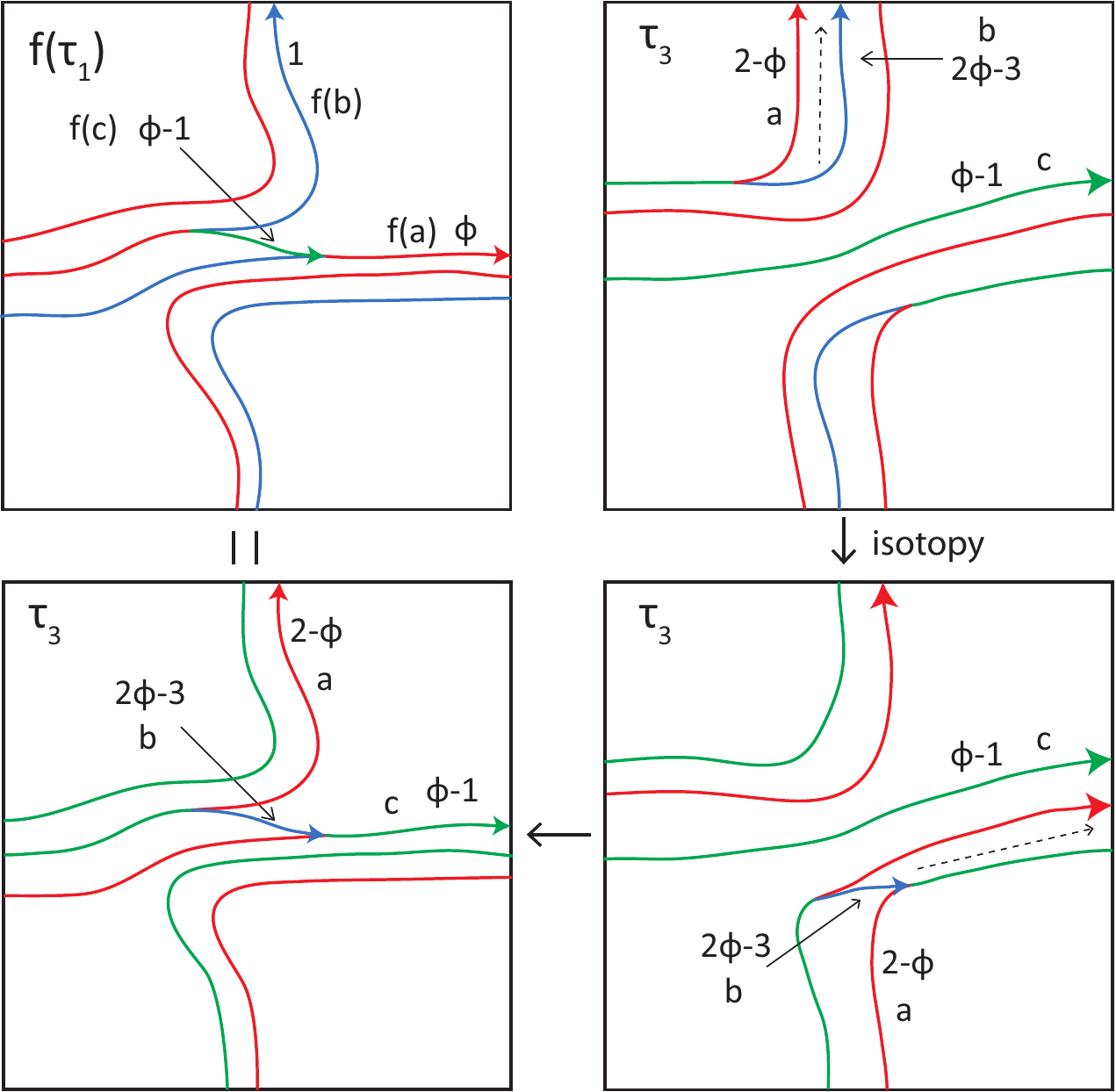}}
\caption{Isotoping $\tau_3$ shows that $f(\tau_1) = \tau_3$. The dotted arrows are aids to visualise the isotopies.}
\label{fig:torus_tt_isotopy}
\end{figure}

We now illustrate how the periodic splitting sequence is detected combinatorially following Section \ref{combeqv}. We compute a homotopy equivalence $\mathfrak{g}_1 : \tau_1 \rightarrow \tau_1$ obtained from $\mathfrak{g}$. First we compute the intermediate function $\mathfrak{g}_1'$.
\begin{eqnarray*}
\mathfrak{g}_1' : \tau_1 &\rightarrow& \tau_0 \\
a &\mapsto& \mathfrak{g}(a) = cacbca \\
b &\mapsto& \mathfrak{g}(cbc) = cbca \\
c &\mapsto& \mathfrak{g}(c) = \mbox{(empty path)}.
\end{eqnarray*}
We now replace every occurrence of $b$ in the image of an edge with $\overline{c}b\overline{c}$.
\begin{eqnarray*}
\mathfrak{g}_1 : \tau_1 &\rightarrow& \tau_1 \\
a &\mapsto& cac(\overline{c}b\overline{c})ca \\
b &\mapsto& c(\overline{c}b\overline{c})ca \\
c &\mapsto& \mbox{(empty path)}.
\end{eqnarray*}
For convenience, we homotope $\mathfrak{g}_1$ to remove backtracking of the form $c\overline{c}$ or $\overline{c}c$ in the image of each edge, so that $\mathfrak{g}_1(a) = caba$ and $\mathfrak{g}_1(b) = ba$. There are two distinct combinatorial equivalences between $\tau_1$ and $\tau_3$, induced by homeomorphisms of $S$ which restrict to maps $\psi_1 : \tau_1 \rightarrow \tau_3$, $\psi_1(a) = c, \psi_1(b) = a, \psi_1(c) = b$, and $\psi_2 : \tau_1 \rightarrow \tau_3$, $\psi_2(a) = c, \psi_2(b) = a, \psi_2(c) = b$, respectively.

We compute $h_1 : \tau_2 \rightarrow \tau_1$, $h_1(a) = a$, $h_1(b) = b$, $h_1(c) = aca$ and $h_2 : \tau_3 \rightarrow \tau_2$, $h_2(a) = bab$, $h_2(b) = b$, $h_2(c) = c$. Let $h = h_1 \circ h_2 : \tau_3 \rightarrow \tau_1$. Then $h\circ \psi_1 : \tau_1 \rightarrow \tau_1$, $a \mapsto aca$, $b \mapsto bab$, $c \mapsto b$.

Let $v \in \tau_1$ be the terminal vertex of $c$. Then $\mathfrak{g}_1$ does not fix $v$. We homotope $\mathfrak{g}_1$ along $c$, so that $\mathfrak{g}_1(a) = \overline{c}(caba)c = abac$, $\mathfrak{g}_1(b) = \overline{c}(ba)c$, and $\mathfrak{g}_1(c)$ is the empty path (notice that we have removed backtracking.) Then $\mathfrak{g}_1$ fixes $v$.

The fundamental group $\pi_1(\tau_1, v)$ is the free group generated by the two loops $A = ac$ and $B = \overline{c}b$. We check that $(\mathfrak{g}_1)_*$ and $(h\circ\psi_1)_*$ are equal as outer automorphisms of $\pi_1(\tau_1, v)$.
We have $(\mathfrak{g}_1)_* : \pi_1(\tau_1, v) \rightarrow \pi_1(\tau_1, v)$, $A = ac \mapsto abac = ABA$, $B = \overline{c}b \mapsto \overline{c}bac = BA$, and $(h\circ\psi_1)_* : \pi_1(\tau_1, v) \rightarrow \pi_1(\tau_1, v)$, $A = ac \mapsto acab = AAB$, $B = \overline{c}b \mapsto \overline{b}bab = AB$. We see that $(\mathfrak{g}_1)_*(\alpha) = A^{-1} (h\circ\psi_1)_*(\alpha) A$ for all $\alpha \in \pi_1(\tau_1, v)$, as required.

It may be checked that the combinatorial equivalence $\psi_2$ does not induce the correct outer automorphism of $\pi_1(\tau_1, v)$, but we do not show the details.

We use the periodic splitting sequence to construct the veering triangulation, see Figure \ref{fig:torus_tt_layers}. We draw the triangulation of the punctured torus $T$ dual to each train track in the periodic splitting sequence, to obtain ``layers" of triangulations of $T$. Between every two layers a tetrahedron is inserted which interpolates between the layers. Layers $0$ and $2$ are then glued by the monodromy.

\begin{figure}[H]
\centerline{\includegraphics[width=\columnwidth]{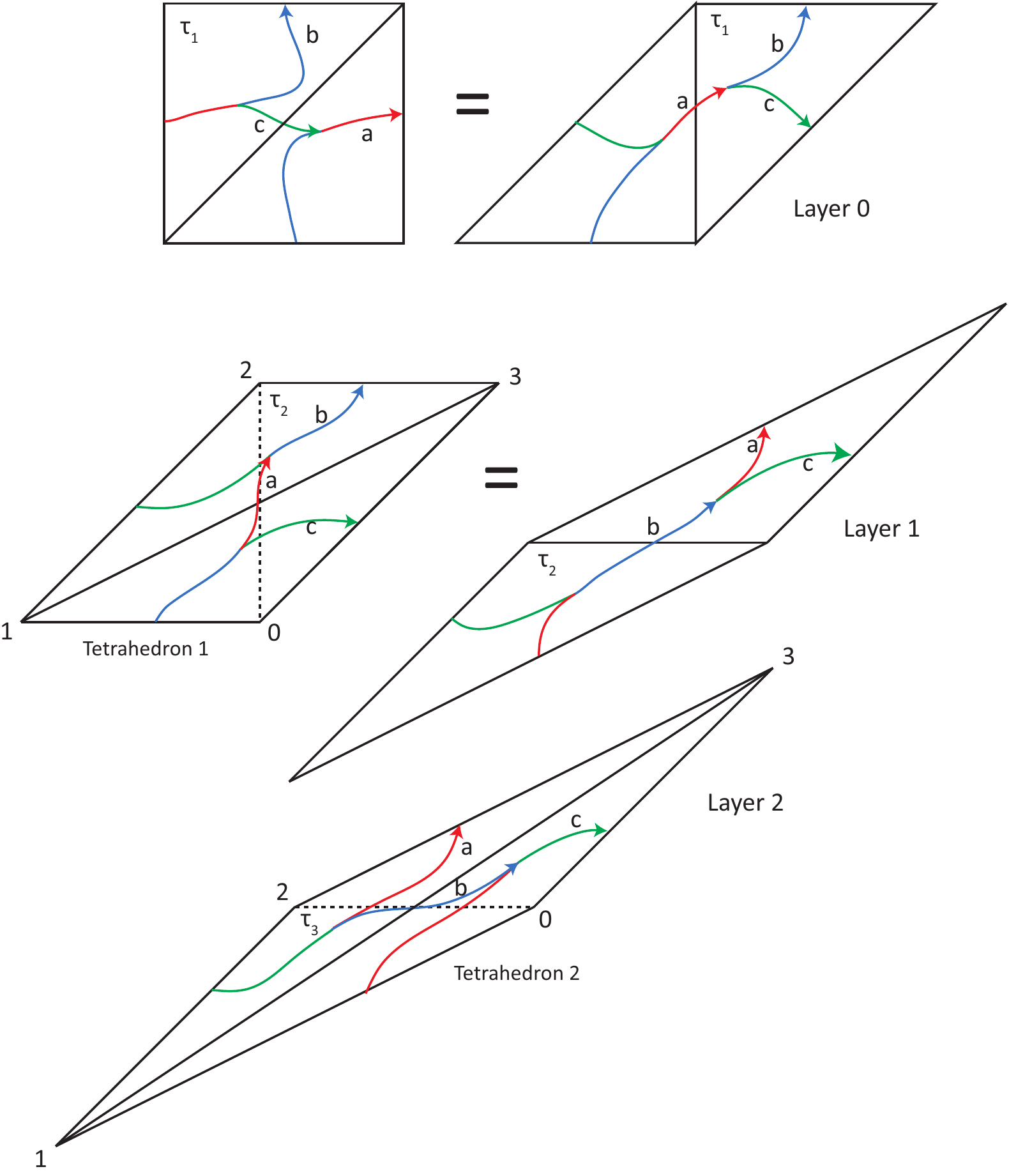}}
\caption{Triangulations of the punctured torus, dual to the train tracks of the splitting sequence. For aesthetic reasons we rearrange the triangles in layers 0 and 1 so that the diagonal of the parallelogram crosses the large branch of a train track.}
\label{fig:torus_tt_layers}
\end{figure}

Let $T_1$ and $T_2$ be tetrahedron $1$ and $2$, respectively.
Since the triangles of the front (back) face of $T_1$ ($T_2$) correspond to triangles of layer $1$, we obtain face pairings between the front of $T_1$ and back face of $T_2$. In order to completely specify all the tetrahedron face pairings it is necessary to pair the triangles of tetrahedron $1$ attached to layer $0$, to the triangles of tetrahedron $2$ attached to layer $2$. 

The edges of triangles in layer $0$ are in one to one correspondence with oriented branches of $\tau_1$ which emanate from the unique vertex of $\tau_1$ interior to the triangle. Since we know that $f(\tau_1) = \tau_3$, with $f(a) = c$, $f(b) = a$ and $f(c) = b$ in agreement with $\psi_1$, we use this information to determine the remaining face pairings. Hence, we have that face $(102)$ of $T_1$ is paired with face $(103)$ of $T_2$, and face $(032)$ of $T_1$ is paired with face $(132)$ of $T_2$.

\begin{center}
\begin{table}[htbp]
  \centering
	\caption{Tetrahedron face pairings}
    \begin{tabular}{|c|c|}
\hline
\mbox{Tetrahedron } 1 & \mbox{Tetrahedron } 2 \\ \hline
(103) & (203) \\
(321) & (021) \\
(102) & (103) \\ 
(032) & (132) \\ 
\hline
\end{tabular}
\end{table}
\end{center}

	\section{Conjugacy testing}\label{sec:conjtest}
	Given pseudo-Anosov homeomorphisms $\phi, \phi' : S \rightarrow S$ of a punctured surface $S$, in this section we describe an algorithm to effectively decide whether or not $\phi$ and $\phi'$ are conjugate in $\mbox{MCG}(S)$. We have implemented this algorithm as part of our computer program.
	
	\begin{mydef} Let $S$ be a punctured surface. Let $\phi, \phi' : S \rightarrow S$ be pseudo-Anosov homeomorphisms with periodic splitting sequences
		$$(\tau_0, \mu_0) \rightharpoonup^m (\tau_m, \mu_m) = \phi(\tau_0, \frac{1}{\lambda}\mu_0) \rightharpoonup \cdots$$
		and
		$$(\tau'_0, \mu'_0) \rightharpoonup^n (\tau'_n, \mu'_n) = \phi'(\tau'_0, \frac{1}{\lambda'}\mu'_0) \rightharpoonup \cdots$$
		respectively, where $n,m \in \Z_{>0}$ and $\lambda, \lambda' \in \R_{>0}$ is the dilatation of $\phi, \phi'$ respectively, as in Theorem \ref{thm:splitting}. We denote the periodic splitting sequences of $\phi, \phi'$ by $\mathcal{S}, \mathcal{S}'$ respectively. We say that $\mathcal{S}$ and $\mathcal{S}'$ are \emph{combinatorially isomorphic} if there exists a diffeomorphism which conjugates between them up to rescaling, that is, $n = m$ and there exists $h : S \rightarrow S$, $p,q \in \Z_{\ge 0}$ and $c \in \R_{>0}$ such that:
		\begin{enumerate}[(i)]
			\item $\phi' = h \circ \phi \circ h^{-1}$, and
			\item $h(\tau_{p+i}, \mu_{p+i}) = (\tau'_{q+i}, c\mu'_{q+i})$, for all $i \in \Z_{\ge 0}$.
		\end{enumerate}
	\end{mydef}
	
	\begin{lemma}[Corollary 3.4 of \cite{MR2866919}]\label{lem:comsplit} Let $\mathcal{L}$ be a measured geodesic lamination suited to train tracks $(\tau, \mu), (\tau', \mu')$. Then $(\tau, \mu)$ and $(\tau', \mu')$ eventually split to a common train track, that is, there exists $(\tau'', \mu'')$ and $n,m\in\Z_{\ge 0}$ such that $(\tau, \mu) \rightharpoonup^n (\tau'', \mu'')$ and $(\tau', \mu') \rightharpoonup^m (\tau'', \mu'')$.
	\end{lemma}
	
	\begin{thm}\label{thm:conj} Let $\phi, \phi' : S \rightarrow S$ be pseudo-Anosov homeomorphisms with periodic splitting sequences
		$\mathcal{S}$ and $\mathcal{S}'$ respectively, as in Theorem \ref{thm:splitting}.
	Then $\phi$ and $\phi'$ are conjugate in $\mbox{MCG}(S)$ if and only if $\mathcal{S}$ and $\mathcal{S}'$ are combinatorially isomorphic.
	\end{thm}
	
	\begin{proof} Assume that $\phi$ and $\phi'$ are conjugate, and let $h \in \mbox{MCG}(S)$ such that $\phi' = h \circ \phi \circ h^{-1}$. 
	
	Let the periodic splitting sequences $\mathcal{S}$ and $\mathcal{S}'$ be given by
	$$(\tau_0, \mu_0) \rightharpoonup^m (\tau_m, \mu_m) = \phi(\tau_0, \frac{1}{\lambda}\mu_0) \rightharpoonup \cdots$$
		and
		$$(\tau'_0, \mu'_0) \rightharpoonup^n (\tau'_n, \mu'_n) = \phi'(\tau'_0, \frac{1}{\lambda'}\mu'_0) \rightharpoonup \cdots$$
		respectively, where $n,m \in \Z_{>0}$ and $\lambda, \lambda' \in \R_{>0}$ is the dilatation of $\phi, \phi'$ respectively.
	
	Let $\mathcal{L}, \mathcal{L}'$ be the stable measured geodesic laminations of $\phi, \phi'$ which are suited to $(\tau_0, \mu_0), (\tau'_0, \mu'_0)$ respectively.
	Then $(\tau'', \mu'') := h(\tau_0, \mu_0)$ is an invariant train track for $\phi'$ and there exists a stable measured geodesic lamination $\mathcal{L}''$ suited to $(\tau'', \mu'')$ (which is essentially given by straightening the leaves of $h(\mathcal{L})$.) By uniqueness of the stable lamination of $\phi'$ in $\mathcal{PML}(S)$, there exists $c \in \R_{>0}$ such that $\mathcal{L}' = c\,\mathcal{L}''$, where $c\,\mathcal{L}''$ denotes the measured geodesic lamination equal to $\mathcal{L}''$ as a lamination but with transverse measure scaled by $c$. Then $\mathcal{L}'$ is suited to $(\tau'', c\mu'')$ and $(\tau'_0, \mu'_0)$. Therefore by Lemma \ref{lem:comsplit}, we know that $(\tau'', c\mu'') = h(\tau_0, c\mu_0)$ and $(\tau'_0, \mu'_0)$ have a common split, that is, there exists $p,q \in \Z_{\ge 0}$ such that
	$h(\tau_p, c\mu_p'') = (\tau'_q, \mu'_q)$, which implies
	$$h(\tau_{p+i}, c\mu_{p+i}'') = (\tau'_{q+i}, \mu'_{q+i}) \mbox{ for }i\in\Z_{\ge 0}.$$
	Thus, $\mathcal{S}$ and $\mathcal{S}'$ are combinatorially isomorphic.
	
	We now prove the other direction. Assume that $\mathcal{S}$ and $\mathcal{S}'$ are combinatorially isomorphic. Then by definition there exists a diffeomorphism $h \in \mbox{MCG}(S)$ such that $\phi' = h \circ \phi \circ h^{-1}$, as required.
	\end{proof}
	
	We remark that Theorem \ref{thm:conj} is similar to \cite[Theorem 10.3.2]{Mosher2} of Lee Mosher.
	
	\subsection{Algorithm for determining conjugacy in \texorpdfstring{$\mbox{MCG}(S)$}{MCG(S)}}
	Proposition \ref{prop:ceqdiff} allows us to determine combinatorially whether or not two periodic splitting sequences are combinatorially isomorphic, which when combined with Theorem \ref{thm:conj} gives an effective algorithm for testing for conjugacy of pseudo-Anosov mapping classes in the mapping class group.
	
	Let $\phi, \phi' : S \rightarrow S$ be pseudo-Anosov homeomorphisms with periodic splitting sequences
		$$(\tau_0, \mu_0) \rightharpoonup^m (\tau_m, \mu_m) = \phi(\tau_0, \lambda^{-1}\mu_0) \rightharpoonup \cdots$$
		and
		$$(\tau'_0, \mu'_0) \rightharpoonup^n (\tau'_n, \mu'_n) = \phi'(\tau'_0, \lambda'^{-1}\mu'_0) \rightharpoonup \cdots$$
	respectively, where $n,m \in \Z_{>0}$ and $\lambda, \lambda' > 1$ is the dilatation of $\phi, \phi'$ respectively, as in Theorem \ref{thm:splitting}.
	We determine whether or not $\phi$ and $\phi'$ are conjugate in $\mbox{MCG}(S)$ using the following procedure.
	\begin{enumerate}[(i)]
	\item If $n \neq m$ or $\lambda \neq \lambda'$ then $\phi$ and $\phi'$ are not conjugate in $\mbox{MCG}(S)$. 
	\item Otherwise, we enumerate all combinatorial equivalences (up to rescaling of measures) between $(\tau_i, \mu_i)$ and $(\tau'_j, \mu'_j)$, for $0 \le i < m = n$ and $0 \le j < n$, then check whether any combinatorial equivalence conjugates between $\phi$ and $\phi'$ (in fact one only needs to check this for $i = 0$ and $0 \le j < n$.)
	
	More precisely, let $\Phi$ be a combinatorial equivalence between $(\tau_i, \mu_i)$ and $(\tau'_j, c\mu'_j)$, where $c \in \R_{>0}$, $0 \le i < n$ and $0 \le j < n$. Then $\Phi$ gives a bijection between the edges of the train tracks which we denote by
	$$\Phi_0 : \mathcal{E}_{dir}(\tau_i) \rightarrow \mathcal{E}_{dir}(\tau'_j),$$
	where the notation $\mathcal{E}_{dir}(\tau_i)$ means the set of directed edges of $\tau_i$.
	Furthermore, by splitting the train tracks $(\tau_i, \mu_i)$ and $(\tau_j, \mu_j)$ simultaneously, keeping track of the bijection between the edges induced by $\Phi_0$, we obtain a combinatorial equivalence
	$$\Phi_n : \mathcal{E}_{dir}(\tau_{i+n}) \rightarrow \mathcal{E}_{dir}(\tau'_{j+n}).$$
	($\Phi_0$ and $\Phi_n$ should be thought of as representing the diffeomorphism induced by $\Phi$ given by Proposition \ref{prop:ceqdiff}, in terms of different underlying train tracks.)
	The homeomorphisms $\phi$ and $\phi'$ induce combinatorial equivalences
	$$\phi_{*} : \mathcal{E}_{dir}(\tau_i) \rightarrow \mathcal{E}_{dir}(\tau_{i+n}) \mbox{ and }\phi'_{*} : \mathcal{E}_{dir}(\tau'_j) \rightarrow \mathcal{E}_{dir}(\tau'_{j+n}),$$
	respectively.
		\begin{enumerate}[(a)]
			\item If $\phi'_{*} \circ \Phi_0 = \Phi_n \circ \phi_{*}$ then we conclude that $\phi$ and $\phi'$ are conjugate in $\mbox{MCG}(S)$, as this implies $\phi' = h \circ \phi \circ h^{-1}$ where $h \in \mbox{MCG}(S)$ is the homeomorphism inducing the combinatorial equivalence $\Phi$ as in Proposition \ref{prop:ceqdiff}.
			\item	Otherwise, if for no combinatorial equivalence does (a) hold, then $\phi$ and $\phi'$ are not conjugate in $\mbox{MCG}(S)$.
		\end{enumerate}
	\end{enumerate}
	
\begin{remark} Given a pseudo-Anosov homeomorphism $\phi$, with periodic splitting sequence of train tracks $(\tau_0, \mu_0) \rightharpoonup^m (\tau_m, \mu_m)$, the measures $\mu_0,\ldots,\mu_m$ are determined up to scaling by the (unmeasured) train tracks $\tau_0,\ldots,\tau_m$ together with the combinatorial equivalence given by $\phi(\lambda^{-1}\tau_0) = \tau_m$, where $\lambda$ is the dilatation of $\phi$. This follows from the argument given in the last paragraph of the proof of Proposition 4.2 in \cite{MR2866919}. Thus, in Step (ii) of our algorithm above, we may weaken our notion of combinatorial equivalence so that we no longer require that a combinatorial equivalence preserves weights of edges of train tracks. As long as the condition in (a) is satisfied the weights are guaranteed to agree. This allows us to avoid comparing weights of edges of train tracks.

Note, however, that in order to compute the train tracks $\tau_0,\ldots,\tau_m$ we need to be able to perform maximal splits, which requires that we identify all edges of a given train track which have maximal weight. In theory, this can be done exactly since the weights of the train tracks can be assumed to be algebraic numbers, as the weights of the initial train track in our algorithm are given by the entries of an eigenvector of an integer valued matrix and thus can be chosen to be algebraic numbers, and weights of subsequent train tracks are obtained from these by algebraic operations. In practice, in our program we only keep track of the weights of train tracks to a certain number of decimal places. Hence, if our program outputs that two homeomorphisms are conjugate then they are guaranteed to be, but if our program outputs that two homeomorphisms are not conjugate then there is a (small) chance that numerical errors occurred and that the homeomorphisms are in fact conjugate.
\end{remark}
	
\section{Results}\label{chapter:results}
We have included the beginning of longer tables of data which we have collected, given as tables in Appendix \ref{app:tables}. A number of triangulations which SnapPy reports are non-geometric were found, and in the case that we start with a pseudo-Anosov homeomorphism of a once-punctured genus 2 surface, a table of such examples is given in Appendix \ref{app:tables}.

Many of the veering triangulations produced are relatively large, and could be simplified by SnapPy to much smaller triangulations, see Figure \ref{fig:tri_table}. Heuristically, triangulations of a fixed hyperbolic manifold with a relatively small number of tetrahedra have a better chance of being geometric, and SnapPy makes an effort to simplify triangulations of link complements that it produces \cite[Chapter 10]{MR2179010}.
\begin{figure}[H]
\centerline{\includegraphics[width=350pt]{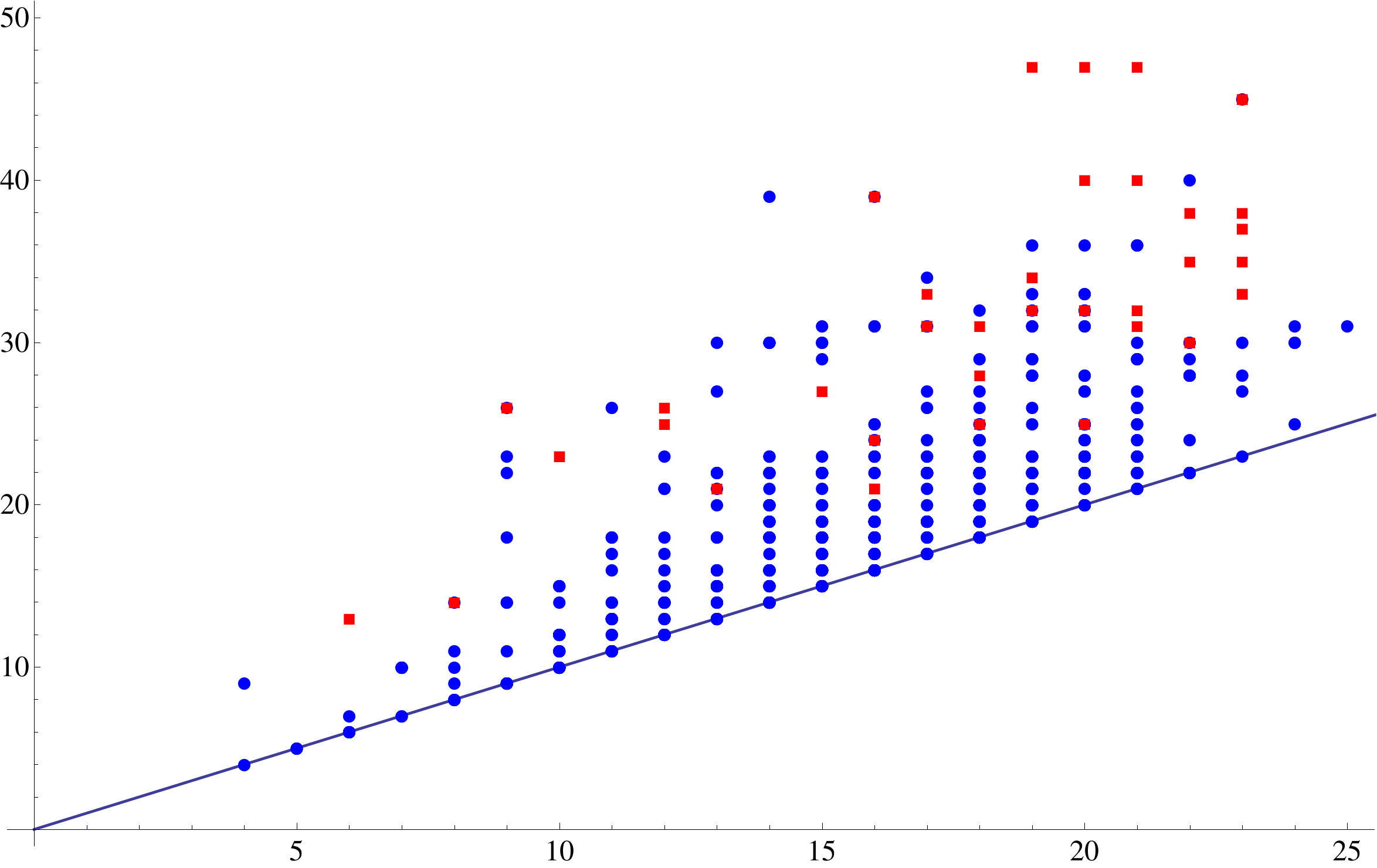}}
\caption{Vertical axis: number of tetrahedra in the veering triangulation. Horizontal axis: number of tetrahedra after the veering triangulation is simplified by SnapPy. Red squares indicate veering triangulations that SnapPy reports are non-geometric.}
\label{fig:tri_table}
\end{figure}

\section{A non-geometric example}\label{sec:nongeo}
In this section we describe the simplest example that we have found of a triangulation
that is non-geometric as reported by SnapPy, and give a rigorous proof that it is indeed non-geometric. This is a 13 tetrahedron triangulation $\mathcal{T}$ coming from the pseudo-Anosov homeomorphism $\varphi = T_{c_1} \circ T_{b_2} \circ T_{a_1} \circ T_{a_1} \circ T_{a_1} \circ T_{b_1} \circ T_{a_1}$ of the once-punctured genus 2 surface, where $T_{\gamma}$ denotes a Dehn twist in the curve $\gamma$ of Figure \ref{fig:twist_curves}. We computed a measured train track $\tau$ suited to the stable geodesic lamination of $\varphi$ which is shown in Figure \ref{fig:nongeo_tt}, where the octagon is punctured at the vertices, and edge identifications are determined by matching arrows on edges.

\begin{minipage}[t]{0.75\textwidth}
\begin{figure}[H]
\includegraphics[height=200pt]{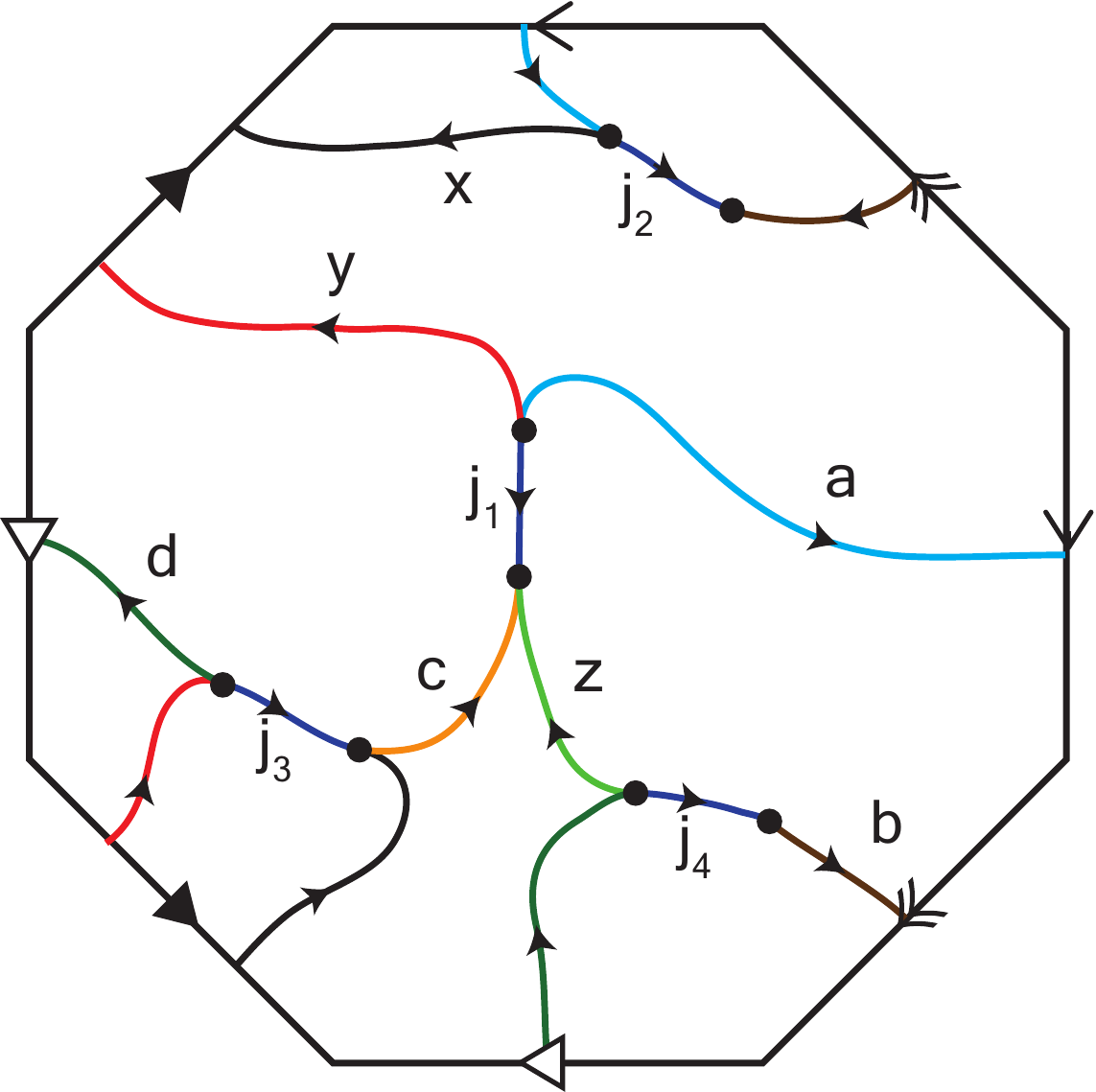}
\caption{Invariant train track $\tau$ obtained by performing the Bestvina-Handel algorithm.}
\label{fig:nongeo_tt}
\end{figure}
\end{minipage}
\begin{minipage}[t]{0.23\textwidth}
Approximate weights of branches:
$$\bordermatrix{ & \cr
								 j_1 & 1.34601 \cr
								 j_2 & 1.46574 \cr
								 j_3 & 1.30459 \cr
								 j_4 & 1.46574 \cr
								 a & 0.50717 \cr
								 b & 1.46574 \cr
								 c & 0.34601 \cr
								 d & 0.46574 \cr
								 x & 0.95857 \cr
								 y & 0.83885 \cr
								 z & 1.00000}.$$
\end{minipage}

A graph map representing $\varphi$ given by applying the Bestvina-Handel algorithm is given by $\mathfrak{g} : \tau \rightarrow \tau$:
\begin{multicols}{3}
\begin{eqnarray*}
a &\mapsto& d\ j_4\ b\ \overline{j_2}\ \overline{a}\ j_1\ \overline{z} \\
b &\mapsto& y\ j_3\ \overline{x}\ j_2\ \overline{b} \\
c &\mapsto& x \\
d &\mapsto& \overline{b}\ \overline{j_4}\ z
\end{eqnarray*}

\begin{eqnarray*}
x &\mapsto& z\ \overline{j_1}\ y\ j_3\ \overline{x}\ j_2\ \overline{b}\ \overline{j_4}\ z\ \overline{j_1}\ a \\
y &\mapsto& d\ j_4\ b \\
z &\mapsto& \overline{c} \\
\end{eqnarray*}

\begin{eqnarray*}
j_1 &\mapsto& j_3 \\
j_2 &\mapsto& j_4 \\
j_3 &\mapsto& \overline{j_2} \\
j_4 &\mapsto& \overline{j_1}
\end{eqnarray*}
\end{multicols}

The dilatation of $\varphi$ is $\lambda = 2.89005..$, which is the largest real root of $x^4 - 2x^3 - 2x^2 -2x + 1$. From the invariant train track we see that the invariant measured foliations have a $6$-prong singular point at the puncture, and that there are no other singular points. The mapping torus $M_\varphi$ is $1$-cusped and has hyperbolic volume $4.85117..$. It is identified as manifold s479 in the SnapPy census of manifolds, which has a geometric triangulation of $M_\varphi$ by $6$ ideal tetrahedra. The triangulation of the boundary torus of $M_\varphi$ is shown in Figure \ref{fig:dtor_btorus}.

\begin{figure}[H]
\centerline{\includegraphics[width=\columnwidth]{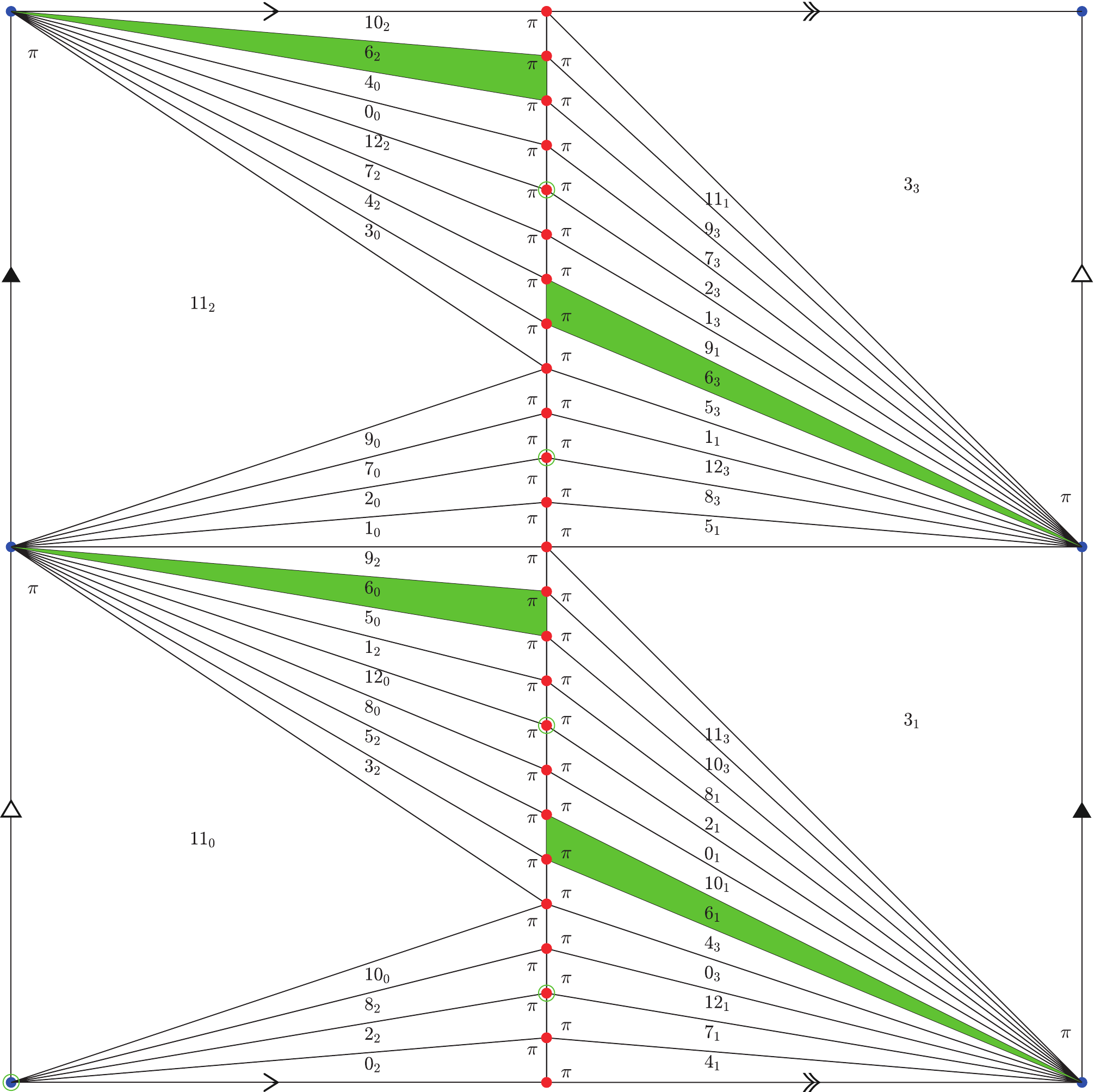}}
\caption{Induced triangulation of boundary torus of the triangulation $\mathcal{T}$. Arrows indicate edge identifications. SnapPy finds a solution to the {\em gluing and completeness equations} (see \cite{ThNotes}) where the shaded triangles correspond to the tetrahedron reported to be negatively oriented. The notation $6_3$ in a triangle means the truncated end of tetrahedron 6 at vertex 3.}
\label{fig:dtor_btorus}
\end{figure}

Generally, we can not immediately rule out the possibility that SnapPy finds a non-geometric solution although a geometric solution exists. We now outline how we rigorously verified that the triangulation $\mathcal{T}$ is non-geometric.

\begin{mydef} Let $\mathcal{T}$ be an ideal triangulation, with tetrahedron edge parameters given by a vector $\vec{z}$. The \emph{algebraic volume} $\mbox{vol}(\vec{z})$ of $\vec{z}$ is the sum of the signed volumes of hyperbolic tetrahedra with shapes given by $\vec{z}$, i.e. where a negatively oriented tetrahedron subtracts its volume from the sum. If it is clear which edge parameters we are referring to, we shall write $\mbox{vol}(\mathcal{T})$ for $\mbox{vol}(\vec{z})$.
\end{mydef}

\begin{thm}\label{thm:nongeo} Let $\mathcal{T}$ be an ideal triangulation of a hyperbolic 3-manifold $M$. Then there exists at most one solution $\vec{z}$ to the {\em gluing and completeness equations} (see \cite{ThNotes}) for $\mathcal{T}$ such that
$$\mbox{vol}(\vec{z}) = \mbox{vol}(M),$$
where $\mbox{vol}(M)$ is the hyperbolic volume of $M$.
\end{thm}
Theorem \ref{thm:nongeo} is a corollary of \cite[Remark 4.1.20]{Stefano} and \cite[Theorem 5.4.1]{Stefano} (see the last dot point of page 6.)

\begin{corol}\label{nongeo} Let $\vec{z}$ be a solution to the gluing and completeness equations for an ideal triangulation $\mathcal{T}$ of a hyperbolic 3-manifold $M$ such that 
\begin{enumerate}
	  \setlength{\itemsep}{1pt}
  \setlength{\parskip}{0pt}
  \setlength{\parsep}{0pt}
		\item $\mbox{vol}(\vec{z}) = \mbox{vol}(M)$, and
		\item $\vec{z}$ is non-geometric, i.e. at least one of the edge parameters has non-positive imaginary part.
	\end{enumerate}
	\vspace{-2mm}
		Then $\mathcal{T}$ is non-geometric, that is, there is no solution to the gluing and completeness equations with all tetrahedra positively oriented.
\end{corol}
\begin{proof} If $\vec{z'}$ is a geometric solution to the gluing and completeness equations, then $\mbox{vol}(\vec{z'}) = \mbox{vol}(M)$. Hence, by Theorem \ref{thm:nongeo} we have $\vec{z} = \vec{z'}$, contradicting that $\vec{z}$ is non-geometric.
\end{proof}

By Corollary \ref{nongeo}, in order to show that $\mathcal{T}$ is non-geometric it suffices to find a non-geometric solution to the gluing and completeness equations which has algebraic volume equal to the volume of $M$. We find an exact non-geometric solution to the gluing and completeness equations using the computer program Snap \cite{Snap}, then verify that it has algebraic volume equal to the volume of $M$.

If $\vec{z}$ consists of the edge parameters to a solution of the gluing and completeness equations then the numbers $z_i$ are in fact algebraic numbers. The computer program Snap attempts to find such a solution, expressing each edge parameter $z_i$ as a polynomial in a number field $\Q(\tau)$, where $\tau$ is an algebraic number. In order to specify $\tau$ exactly, Snap provides the minimal polynomial $m(x)$ of $\tau$ as well as calculates $\tau$ to sufficiently many decimal places to uniquely specify it as the root of $m(x)$ closest to the decimal approximation. See \cite{MR1758805} for more information about Snap.

\begin{mydef}\cite[\S 3.1]{GenDef} Let $\mathcal{T}$ be a triangulation. A \emph{Pachner 2-3 move} is a move at a common face of two tetrahedra in $\mathcal{T}$ which produces the triangulation $\mathcal{T}'$, obtained by removing the face and inserting a dual edge, see Figure \ref{fig:pachner_moves}. If the tetrahedra in $\mathcal{T}$ are assigned edge parameters, then the three new tetrahedra in $\mathcal{T}'$ are assigned edge parameters as shown in Figure \ref{fig:pachner_moves}. A \emph{Pachner 3-2 move} is the reverse of a Pachner 2-3 move, that is, an edge with three surrounding tetrahedra is replaced by a dual face.
\end{mydef}

\begin{figure}[H]
\centerline{\includegraphics[height=180pt]{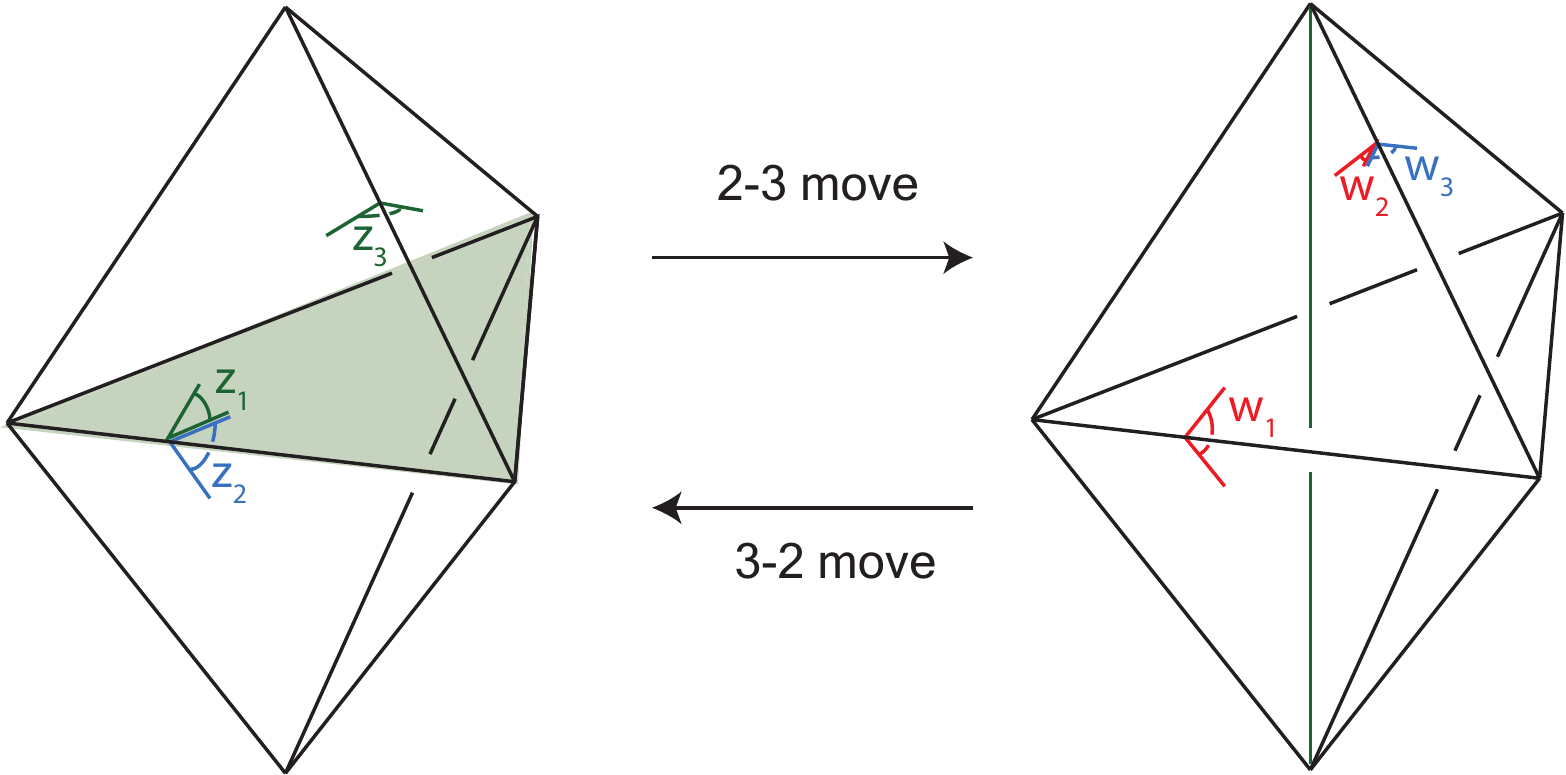}}
\caption{Pachner 2-3 and 3-2 moves. On the left two tetrahedra have a common face and on the right three tetrahedra have a common edge, shown in green. The edge parameters are related by: $z_3 = w_2 w_3$ and $w_1 = z_1 z_2$, and similar relations at the other edges allow us to determine the shapes of each tetrahedron.}
\label{fig:pachner_moves}
\end{figure}

\begin{thm}\label{thm:volpres} Let $M$ be a hyperbolic 3-manifold and let $\mathcal{T}$ be an ideal triangulation of $M$. Let $\vec{z}$ be a solution to the gluing and completeness equations for $\mathcal{T}$, where we allow negatively oriented and flat tetrahedra, i.e. $z_i \in \C\backslash\{0,1\}$ for each edge parameter $z_i$. Let $\mathcal{T}'$ be the ideal triangulation obtained from $\mathcal{T}$ by a 2-3 or 3-2 Pachner move, and let $\vec{z'}$ be the corresponding edge parameters. If the edge parameters $\vec{z'}$ define non-degenerate tetrahedra, i.e. no edge parameter is equal to $0, 1$ or $\infty$, then the algebraic volumes of $\vec{z}$ and $\vec{z'}$ are equal.
\end{thm}

This theorem is a consequence of the ``five-term relation," an identity of the dilogarithm function, see paragraph 2 in the proof of Proposition 10.1 in \cite{MR1663915} for further discussion.

\begin{prop}\label{prop:nongeo} Let $\varphi : S \rightarrow S$ be the pseudo-Anosov homeomorphism $\varphi = T_{c_1} \circ T_{b_2} \circ T_{a_1} \circ T_{a_1} \circ T_{a_1} \circ T_{b_1} \circ T_{a_1}$ of the once-punctured genus 2 surface. Let $\mathcal{T}$ be the veering triangulation of the mapping torus $M_\varphi$ with respect to $\varphi$. Then $\mathcal{T}$ is non-geometric.
\end{prop}

We outline the steps we took to verify Proposition \ref{prop:nongeo}.
We found a sequence $$\mathcal{T} \stackrel{2-3}{\longrightarrow} \mathcal{T}_1 \stackrel{3-2}{\longrightarrow} \mathcal{T}_2 \stackrel{3-2}{\longrightarrow} \mathcal{T}_3 \stackrel{3-2}{\longrightarrow} \mathcal{T}_4,$$
of Pachner 2-3 and 3-2 moves starting with an exact non-geometric solution to Agol's triangulation $\mathcal{T}$, as given by Snap. We verified that the edge parameters of $\mathcal{T}_4$ are all positively-oriented, so that $\mathcal{T}_4$ is geometric and thus $\mbox{vol}(\mathcal{T}_4)$ equals the hyperbolic volume of $M_\varphi$. Furthermore, for $i=1,2,3$, we checked that none of the tetrahedron shapes of $\mathcal{T}_i$ are degenerate. Hence, by Theorem \ref{thm:volpres}, $\mbox{vol}(\mathcal{T}) = \mbox{vol}(\mathcal{T}_4)$. Therefore $\mbox{vol}(\mathcal{T})$ equals the hyperbolic volume of $M_\varphi$. Thus, by Corollary \ref{nongeo} Agol's triangulation $\mathcal{T}$ is non-geometric, as required.

\begin{remark} We expect that all the other examples in Tables \ref{table:nongeo1} and \ref{table:nongeo2} of Appendix \ref{app:tables} are also non-geometric. However, we have not shown this rigorously in the other cases.
\end{remark}

\section{Further work}
Although we have found that veering triangulations are not always geometric, we may ask:

{\bf Question:} Given a veering triangulation coming from Agol's construction, can we find positively-oriented hyperbolic ideal tetrahedra shapes for a (possibly incomplete) hyperbolic structure?

This is equivalent to finding a solution to the gluing equations where every edge parameter has positive imaginary part. Such a solution corresponds to a point in Dehn surgery space. Figure \ref{fig:dtor[5,4,5,5,5,2,3].tri} is obtained by using SnapPy to numerically solve the gluing and $(p,q)$-Dehn surgery equations and colouring each point $(p,q)$. A point is coloured green when a solution with all tetrahedra positively oriented is found, and blue when SnapPy finds a solution where at least one tetrahedron is negatively oriented. We see that even in this non-geometric example SnapPy finds positively-oriented hyperbolic ideal tetrahedra shapes for incomplete structures of the bundle. The complete hyperbolic structure corresponds to the point at infinity.

\begin{figure}[H]
\centerline{\includegraphics[height=250pt]{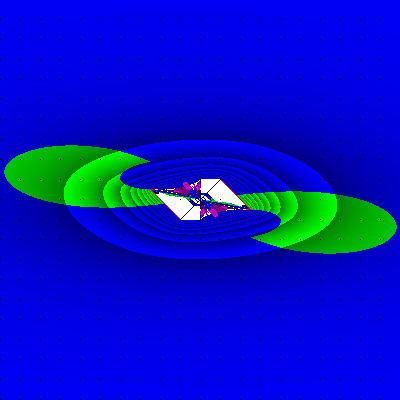}}
\caption{Dehn surgery space for $13$ tetrahedron non-geometric example.}
\label{fig:dtor[5,4,5,5,5,2,3].tri}
\end{figure}

\appendix

\section{Computing the graph map}\label{sec:comp_graph_map}
Let $S$ be a surface of genus $g \ge 0$ with $p > 0$ punctures. Let $f : S \rightarrow S$ be a homeomorphism permuting the punctures, given by a composition $$f = T_{n} \circ \cdots \circ T_{2} \circ T_{1},$$
where each of $T_1, \ldots, T_n$ is either a left or right Dehn twist in one of the curves shown in Figure \ref{fig:twist_curves}, or a half-twist permuting adjacent punctures $i$ and $i+1$, where $i \in \{1,2,\ldots,p-1\}$ (see Figure \ref{fig:half_twist}). The Dehn twists in curves shown in Figure \ref{fig:twist_curves} and half-twists in adjacent punctures generate the mapping class group of $S$.

Let $H \subseteq S$ be a graph homotopy equivalent to $S$. Then $f$ induces a homotopy equivalence $\mathfrak{g} : H \rightarrow H$ and conversely the isotopy class of $f$ is uniquely determined by $\mathfrak{g}$. By a homotopy if necessary, we assume that $\mathfrak{g}$ maps vertices to vertices and oriented edges to edge paths, i.e. that $\mathfrak{g}$ is a graph map representing $f$.  In this section we briefly outline how we compute such a map $\mathfrak{g}$ representing $f$, which is the starting point for the Bestvina-Handel algorithm executed by Trains. 

We compute $\mathfrak{g}$ in two steps, first we obtain a map $\mathfrak{g}' : G \rightarrow G$, where $G \subset S$ is the graph shown in Figure \ref{fig:graph_on_surface}. Note that some complementary regions of $G$ are disks, rather than punctured disks, so that $G$ is not homotopy equivalent to $S$ and $\mathfrak{g}'$ is not a graph map representing $f$ in the usual sense. Our next step will be to modify $G$, obtaining a graph $H \subset G$ homotopy equivalent to $S$ and to modify $\mathfrak{g}'$ to obtain a graph map $\mathfrak{g} : H \rightarrow H$ representing $f$. The reason we start with this larger graph $G$ is that the action of our chosen generators of the mapping class group $\mbox{MCG}(S)$ on $G$ is easier to compute.

\begin{figure}[H]
\centerline{\includegraphics[width=\textwidth]{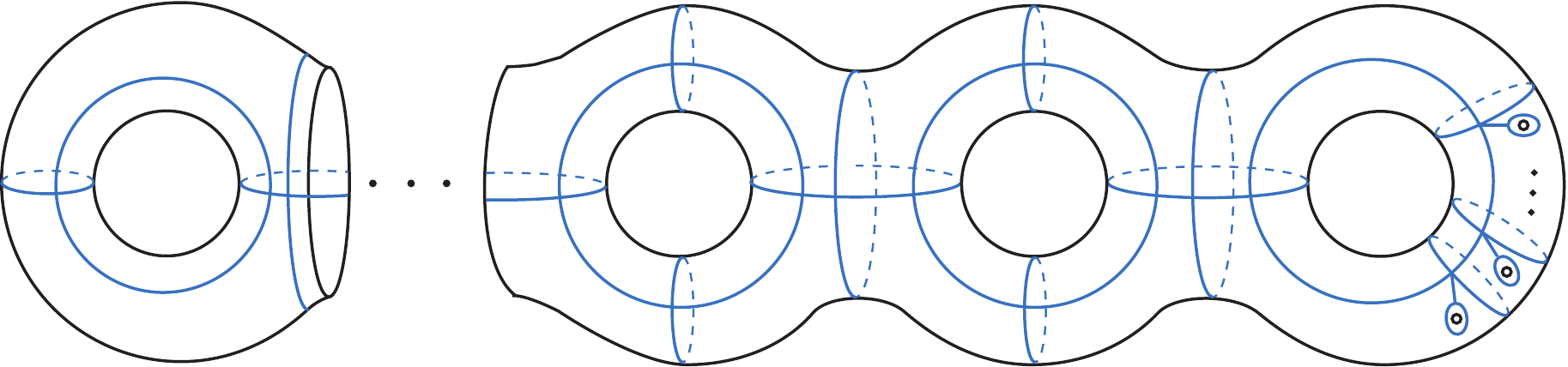}}
\caption{The vertices of the graph are the intersections of the blue curves.}
\label{fig:graph_on_surface}
\end{figure}
In practice, for convenience, instead of $G$ we actually begin with a graph slightly smaller than $G$, given by collapsing some of the edges of $G$, however for clarity we illustrate the essentially identical procedure with $G$.

For $i \in \{1,\ldots,n\}$, we can homotope $T_i$ so that it maps $G$ into $G$ with vertices mapping to vertices and edges mapping to edge paths. The resulting map restricts to a map on $G$. We define $\widehat{T}_i : G \rightarrow G$ to be some choice of such a map (any choice will do.) Then $\mathfrak{g}' : G \rightarrow G$ is given by $\mathfrak{g}' = \widehat{T}_{n} \circ \cdots \circ \widehat{T}_{2} \circ \widehat{T}_{1}$. More precisely, assume first that $T_i$ is a Dehn twist in some curve $\gamma$ of Figure \ref{fig:twist_curves}, supported in a small closed regular neighbourhood $N(\gamma)$ of $\gamma$. We identify $\gamma$ with the natural choice of loop in $G$. Let $h : S \rightarrow S$ be a map homotopic to the identity which retracts the annulus $N(\gamma)$ onto the circle $\gamma$ and is the identity outside of a small open regular neighbourhood of $N(\gamma)$, so that $h \circ T_i$ restricts to a map $\widehat{T}_i := h \circ T_i : G \rightarrow G$ as required. If $T_i$ is an anticlockwise half twist in punctures $j$ and $j+1$ then we can homotope $T_i$ so that the image of $G$ is as shown in Figure \ref{fig:punc_perm_graph}. Then we can further homotope $T_i$ so that it collapses edges drawn as parallel in Figure \ref{fig:punc_perm_graph} onto the corresponding edge of $G$, so that $T_i$ restricts to a map $\widehat{T}_i : G \rightarrow G$ as required. The case where $T_i$ is a clockwise half twist is similar.

\begin{figure}[H]
\centerline{\includegraphics[width=\textwidth]{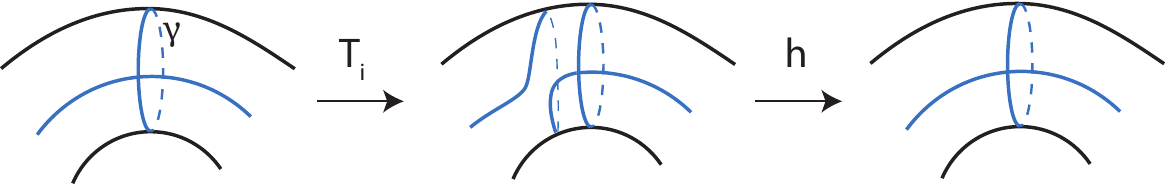}}
\caption{}
\label{fig:dehn_twist}
\end{figure}

\begin{figure}[H]
\centerline{\includegraphics[height=150pt]{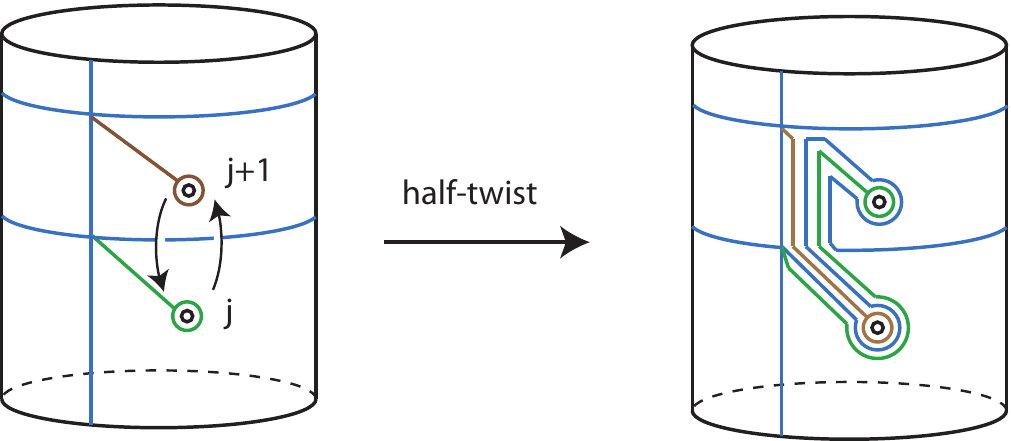}}
\caption{}
\label{fig:punc_perm_graph}
\end{figure}

Let $R$ be a complementary region of $G \subset S$ which is homeomorphic to a disk. Then $\partial R$ can be represented as a cycle $C = e_1, e_2, \ldots, e_m$ of edges (Figure \ref{fig:comp_region_uni_cover}.) Let $e$ be an edge which appears exactly once in $C$ such that $\overline{e}$ is not in $C$. Note that such an $e$ exists since if every edge of $C$ appeared in pairs then gluing along the boundary of $R$ would give a closed surface, contrary to the fact that $S$ is a connected punctured surface. Without loss of generality, assume that $e = e_1$. Let $G_1 = G\backslash e$ be the graph obtained by removing $e$ (including any degree 1 vertex incident with $e$.)

\begin{figure}[H]
\centerline{\includegraphics[height=150pt]{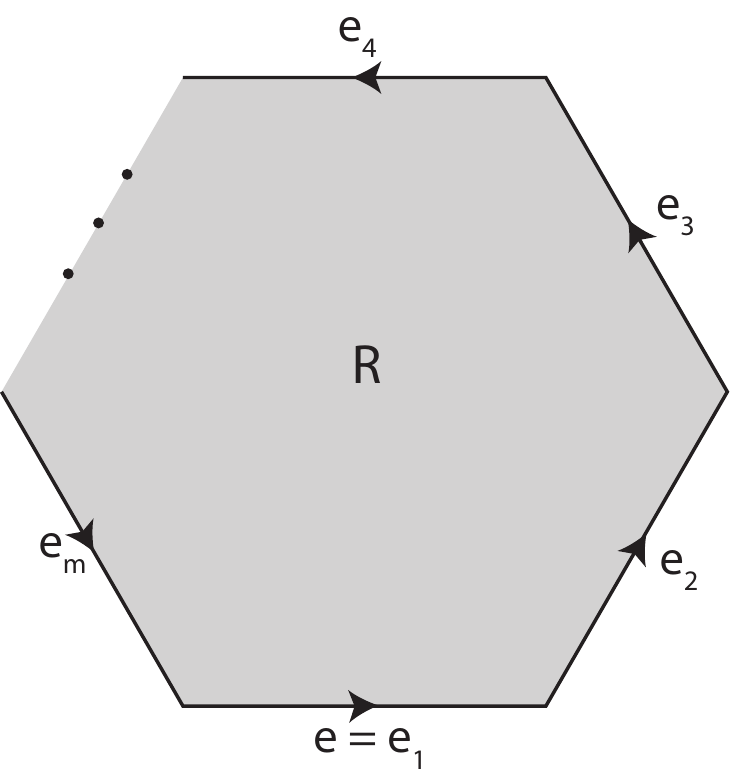}}
\caption{}
\label{fig:comp_region_uni_cover}
\end{figure}

There exists a map $h : S \rightarrow S$ homotopic to the identity such that $h(e) = \overline{e}_m\,\overline{e}_{m-1}\cdots \overline{e}_1$, given by homotoping $e$ across $R$. The map $h$ restricts to a map $h : G_1 \rightarrow G_1$ and $h \circ \mathfrak{g}'_1: G \rightarrow G_1$ restricts to a map $\mathfrak{g}'_1 = \mathfrak{g}' \circ h : G_1 \rightarrow G_1$. 

By repeating the above procedure on the pair $(\mathfrak{g}'_1, G_1)$ to obtain $(\mathfrak{g}'_2, G_2)$ and so on we eventually obtain a graph map $\mathfrak{g}_k' : G_k \rightarrow G_k$ representing $f$, where $k \in \Z_{>0}$ and $G_k \subset G$ is homotopy equivalent to $S$. Note that each time this procedure is performed the number of complementary regions homeomorphic to a disk is reduced, so this process eventually terminates.

\begin{remark}The homeomorphism $f$ permutes the punctures of $S$. Let $n \in \Z_{>0}$ be the number of orbits of punctures under $f$. Notice that in our initial graph $G$ there is a peripheral loop bounding each puncture. The Bestvina-Handel algorithm, which Trains implements, requires that the graph map given as input is on a graph which contains a peripheral loop about each puncture in $n-1$ orbits of punctures and that these peripheral loops are disjoint. We arbitrarily choose $n-1$ orbits of punctures of $G$ and when we iteratively perform the above procedure we ensure that the peripheral loops about punctures in these orbits are not removed.
\end{remark}

\section{Checking a triangulation for a veering structure}\label{sec:veering_check}

Given a triangulation with a taut angle structure, it is very easy to check whether or not it can be given consistent coorientations: Choose a starting tetrahedron $T$, and choose one of the two possible coorientations for that tetrahedron. This determines coorientations for the faces of the neighbours of $T$. We continue assigning coorientations to the tetrahedra, spreading through the tetrahedra of the triangulation. We are able to finish this process without finding contradictory coorientations on a tetrahedron if and only if the taut angle structure can be promoted to a taut structure.

For the experimental results given in \cite{MR2860987} we used an algorithm to find any veering structures on a given triangulation. This was also useful as a check to make sure that the algorithm given in Section \ref{chapter:implementation} does indeed produce veering triangulations. The techniques developed by Ben Burton and implemented in Regina \cite{regina} can likely be modified to search for veering structures very efficiently. However, we used a more naive procedure, which still seems to be very fast.

Each tetrahedron in a triangulation can potentially be taut (without coorientations) in three different ways, corresponding to the three pairs of opposite edges that can have $\pi$ angles assigned to them. Essentially, we do a brute force search through the $3^N$ possible choices of taut structure on each of the $N$ tetrahedra in the triangulation. However, we can very effectively prune this search when we are looking for a veering structure, because the veering condition is so stringent. 

We consider the tree of possible partial veering structures, where the root of the tree has no taut structure on any of the tetrahedra, and the leaves of the tree each have a taut structure assigned to every tetrahedron. We start at the root and traverse towards the leaves, checking at each step if
\begin{enumerate} 
\item we have an inconsistent colouring of edges, or 
\item we have determined the taut structure at every tetrahedron incident to an edge, but with a sum of angles at that edge that is not equal to $2\pi$.
\end{enumerate}
The second of these two conditions requires many tetrahedra to have determined taut structures, but the first can show an immediate contradiction with only two tetrahedra, or even one given self-identifications. Given such a contradiction we search no further along this branch of the tree. This vastly cuts down the search space, particularly if the tetrahedra (and so the levels of the tree) are ordered so that tetrahedra that are close to each other (or consecutive) in the list are close to each other (or incident) in the triangulation. Running on a laptop (Macbook Pro, 2.5 GHz processor), the algorithm checked a 330 tetrahedron triangulation and found the one veering structure in 90 minutes. 

\section{Tables}\label{app:tables}
We enumerated all pseudo-Anosov homeomorphisms of the once-punctured genus two surface (up to conjugation and inversion in the mapping class group) which can be expressed as a composition of at most $7$ Dehn twists in the curves $a_1, b_1, b_2, c_1$ and $c_2$ of Figure \ref{fig:twist_curves}. Note that we have not included Dehn twists in $e_1$; these can be expressed in terms of Dehn twists in the curves chosen. Tables \ref{table:s1} and \ref{table:s2} include the beginning of the complete list of data we have available; the complete list contains $603$ mapping classes. The tables are sorted first by hyperbolic volume, then by dilatation (both up to $5$ decimal places.) We have included a complete list of the mapping classes for which SnapPy reports the triangulation is non-geometric; these are given in Tables \ref{table:nongeo1} and \ref{table:nongeo2}.

Let $\varphi : S \rightarrow S$ be a pseudo-Anosov homeomorphism of a once-punctured genus two surface, and let $M = M_{\varphi^\circ}$, where $\varphi^\circ$ is the restriction of $\varphi$ to the surface with all singular points of the invariant foliations removed. As a check, for each $\varphi$ in our table below, using SnapPy we perform Dehn fillings using our triangulation of $M_{\varphi^\circ}$ to obtain the manifold $M_{\varphi}$, which we verified agrees (up to isometry) with the computation of $M_{\varphi}$ given by the Twister module \cite{Twister} of SnapPy. We also checked that the triangulations produced are veering (this is described in Appendix \ref{sec:veering_check}.)

\begin{table}[htbp]
  \centering
  \caption*{Column key for the following tables.}
    \begin{tabular}{|l|l|}
		\hline
    Dehn twist word & \multicolumn{1}{m{12cm}|}{Dehn twists representing $\varphi$, e.g. b1 C1 C2 B2 represents the mapping class given by a left Dehn twist in $b_1$, followed by right Dehn twist in $c_1$ and so on. Lowercase letters indicate left Dehn twists and uppercase indicate right Dehn twists.} \\ \hline
    Isom. class & Hyperbolic isometry class of $M$. \\ \hline
    Growth rate & Dilatation of $\varphi$. \\ \hline
    Volume & \multicolumn{1}{m{12cm}|}{Hyperbolic volume of $M$. An asterisk following the volume indicates that SnapPy reports that Agol's veering triangulation of $M$ is non-geometric.}\\ \hline
    \#Sing. & Number of singular points in invariant foliation of $\varphi$. \\ \hline
    \#Cusps & Number of cusps of $M$. \\ \hline
    \#Tet. & Number of tetrahedra in Agol's triangulation. \\ \hline
    Prongs. & \multicolumn{1}{m{12cm}|}{Number of prongs at each singular point. The first value in the list is the number of prongs at the puncture of $S$.}\\ \hline
    Sing. perm & \multicolumn{1}{m{12cm}|}{Permutation of singular points with respect to the order listed in the column \emph{Prongs.}, where e.g. $1, 2, 0$ represents the permutation $0\rightarrow 1$, $1 \rightarrow 2$, $2 \rightarrow 0$.} \\ \hline
    Sing. to cusp & \multicolumn{1}{m{12cm}|}{Cusp number of respective singular point, i.e. indicating which singular points are identified as a single cusp of $M$. For example, 0, 0, 1 would indicate that singular points 0 and 1 form one cusp, and singular point 2 forms another cusp of $M$.} \\ \hline
    \end{tabular}%
  \label{tab:addlabel0}%
\end{table}%

\begin{landscape}
\begin{table}[htbp]\small
  \centering
		\caption{Start of list of examples for once-punctured genus 2 surface.}
    \begin{tabular}{|l|l|l|l|l|l|l|l|l|l|l|}
		\hline
    Dehn twist word & \specialcell{Isom.\\class} & \specialcell{Growth\\rate} & Volume & \#Sing. & \#Cusps & \#Tet. & Prongs & Sing. perm & \specialcell{Sing. to\\cusp} \\
		\hline
    a1 a1 a1 b1 c1 b2 & 486   & 1.72208 & 3.17729 & 1     & 1     & 4     & 6   & 0   & 0 \\
    a1 a1 a1 b1 c1 b2 c2 & 485   & 1.88320 & 4.05977 & 2     & 2     & 9     & 4, 4 & 0, 1 & 0, 1 \\
    a1 C1 B1 B1 A1 A1 B2 & 177   & 2.08102 & 4.46466 & 1     & 1     & 5     & 6   & 0   & 0 \\
    a1 b1 c1 b1 b1 a1 b2 & 441   & 2.29663 & 4.46466 & 1     & 1     & 7     & 6   & 0   & 0 \\
    a1 b1 C1 C1 C2 B2 C1 & 392   & 1.88320 & 4.74950 & 2     & 2     & 5     & 1, 7 & 0, 1 & 0, 1 \\
    a1 a1 a1 a1 b1 c1 b2 & 487   & 2.89005 & 4.85117* & 1     & 1     & 13    & 6   & 0   & 0 \\
    b1 C1 B1 B2 C1 C2 & 32    & 2.36921 & 5.04490 & 1     & 1     & 6     & 6   & 0   & 0 \\
    a1 b1 C1 B2 C1 C2 & 388   & 1.72208 & 5.33349 & 4     & 3     & 6     & 1, 3, 4, 4 & 0, 1, 3, 2 & 0, 1, 2, 2 \\
    b1 C1 C2 B2 & 21    & 2.15372 & 5.33349 & 3     & 2     & 6     & 2, 4, 4 & 0, 2, 1 & 0, 1, 1 \\
    b1 C1 C2 B2 C1 & 21    & 2.29663 & 5.33349 & 2     & 2     & 6     & 2, 6 & 0, 1 & 0, 1 \\
    a1 B1 A1 C1 B1 a1 B2 & 327   & 2.61803 & 5.33349 & 1     & 1     & 7     & 6   & 0   & 0 \\
    a1 b1 b1 c1 c2 b2 & 455   & 2.29663 & 6.02305 & 2     & 2     & 10    & 4, 4 & 0, 1 & 0, 1 \\
    a1 B1 C1 B1 B2 C2 & 250   & 1.91650 & 6.35459 & 6     & 3     & 10    & 1, 3, 3, 3, 3, 3 & 0, 1, 5, 2, 3, 4 & 0, 1, 2, 2, 2, 2 \\
    b1 c1 c1 b2 c2 c2 & 122   & 2.89005 & 6.35459 & 1     & 1     & 7     & 6   & 0   & 0 \\
    b1 c2 B2 C1 B1 B1 c2 & 44    & 3.09066 & 6.35459 & 3     & 2     & 10    & 2, 4, 4 & 0, 2, 1 & 0, 1, 1 \\
    b1 b1 C1 C2 B2 C1 & 44    & 3.25426 & 6.35459 & 2     & 2     & 10    & 2, 6 & 0, 1 & 0, 1 \\
    a1 B1 C1 B2 C2 & 219   & 1.96355 & 6.36674 & 6     & 2     & 9     & 1, 3, 3, 3, 3, 3 & 0, 5, 1, 2, 3, 4 & 0, 1, 1, 1, 1, 1 \\
    a1 b1 C1 B2 C2 & 386   & 1.68491 & 6.55174 & 6     & 3     & 10    & 1, 3, 3, 3, 3, 3 & 0, 2, 1, 5, 3, 4 & 0, 1, 1, 2, 2, 2 \\
    a1 b1 C1 B2 B2 C2 & 386   & 1.97482 & 6.55174 & 4     & 3     & 10    & 1, 3, 3, 5 & 0, 2, 1, 3 & 0, 1, 1, 2 \\
    b1 c1 b2 C2 C2 C2 & 110   & 4.05624 & 6.75519 & 3     & 2     & 14    & 2, 4, 4 & 0, 2, 1 & 0, 1, 1 \\
    b1 b1 c1 b2 C2 C2 C2 & 110   & 4.21208 & 6.75519 & 2     & 2     & 14    & 2, 6 & 0, 1 & 0, 1 \\
    a1 B1 A1 C1 B1 B2 C2 & 323   & 2.61803 & 6.92738 & 2     & 2     & 11    & 4, 4 & 0, 1 & 0, 1 \\
    b1 c1 b2 C2 C2 C2 C2 & 109   & 5.03756 & 6.95235 & 3     & 2     & 18    & 2, 4, 4 & 0, 2, 1 & 0, 1, 1 \\
    a1 b1 b1 c1 b2 c2 b2 & 457   & 2.61803 & 7.16394 & 2     & 2     & 9     & 4, 4 & 0, 1 & 0, 1 \\
    a1 a1 b2 c1 b1 b1 & 476   & 4.13016 & 7.16394* & 1     & 1     & 14    & 6   & 0   & 0 \\
    a1 a1 b1 C1 B2 C1 C2 & 477   & 1.97482 & 7.32772 & 5     & 4     & 8     & 1, 3, 3, 3, 4 & 0, 3, 2, 1, 4 & 0, 1, 2, 1, 3 \\
			\hline
    \end{tabular}%
  \label{table:s1}%
\end{table}%
\end{landscape}

\begin{landscape}
\begin{table}[htbp]\small
  \centering
		\caption{Continuation of Table \ref{table:s1}}
    \begin{tabular}{|l|l|l|l|l|l|l|l|l|l|l|}
		\hline
    Dehn twist word & \specialcell{Isom.\\class} & \specialcell{Growth\\rate} & Volume & \#Sing. & \#Cusps & \#Tet. & Prongs & Sing. perm & \specialcell{Sing. to\\cusp}\\
		\hline
    a1 b1 c1 c1 b2 c2 & 437   & 2.61803 & 7.32772 & 2     & 2     & 10    & 4, 4 & 0, 1 & 0, 1 \\
    a1 B1 A1 C1 B1 B2 B2 & 324   & 3.34697 & 7.41416 & 1     & 1     & 8     & 6   & 0   & 0 \\
    a1 b1 c1 c1 c1 b2 & 439   & 4.30933 & 7.50977 & 1     & 1     & 14    & 6   & 0   & 0 \\
    a1 B1 c1 b2 c1 c2 & 359   & 3.50607 & 7.56148 & 2     & 2     & 9     & 2, 6 & 0, 1 & 0, 1 \\
    a1 b1 c1 B1 c1 b1 b2 & 430   & 4.13016 & 7.60355 & 1     & 1     & 11    & 6   & 0   & 0 \\
    b1 c1 c1 c2 c2 c2 b2 & 121   & 3.73205 & 7.77234 & 1     & 1     & 8     & 6   & 0   & 0 \\
    a1 b1 c1 B2 C2 B2 & 416   & 1.96355 & 7.86790 & 5     & 3     & 9     & 1, 3, 3, 3, 4 & 0, 3, 1, 2, 4 & 0, 1, 1, 1, 2 \\
    a1 b1 A1 b1 c1 B2 C2 & 416   & 2.08102 & 7.86790 & 4     & 3     & 9     & 1, 5, 3, 3 & 0, 1, 3, 2 & 0, 1, 2, 2 \\
    b1 C1 B1 B1 B2 C1 C2 & 33    & 5.27451 & 7.94511 & 1     & 1     & 15    & 6   & 0   & 0 \\
    a1 B1 c1 b2 c2 & 357   & 3.44148 & 7.97974 & 3     & 2     & 10    & 2, 4, 4 & 0, 2, 1 & 0, 1, 1 \\
    a1 b1 b1 C1 B2 C2 C2 & 445   & 2.08102 & 8.00023 & 4     & 3     & 9     & 1, 4, 3, 4 & 0, 3, 2, 1 & 0, 1, 2, 1 \\
    b1 C2 B2 C1 C1 & 7     & 2.61803 & 8.00023 & 3     & 2     & 9     & 2, 4, 4 & 0, 2, 1 & 0, 1, 1 \\
    b1 b1 c1 b2 b2 C2 & 7     & 2.89005 & 8.00023 & 2     & 2     & 9     & 2, 6 & 0, 1 & 0, 1 \\
    a1 b1 c1 b2 C2 C2 & 432   & 2.98307 & 8.00577 & 6     & 2     & 17    & 1, 3, 3, 3, 3, 3 & 0, 5, 1, 2, 3, 4 & 0, 1, 1, 1, 1, 1 \\
    a1 B1 C1 B1 B2 c2 & 254   & 2.96557 & 8.11953 & 2     & 2     & 8     & 2, 6 & 0, 1 & 0, 1 \\
    a1 a1 B1 C1 B1 B2 C2 & 469   & 2.94699 & 8.19570 & 6     & 3     & 22    & 1, 3, 3, 3, 3, 3 & 0, 1, 5, 2, 3, 4 & 0, 1, 2, 2, 2, 2 \\
    a1 a1 b1 b1 c1 c2 b2 & 482   & 4.21208 & 8.24198* & 2     & 2     & 26    & 4, 4 & 0, 1 & 0, 1 \\
    a1 a1 a1 b2 c1 b1 b1 & 483   & 6.11129 & 8.25157* & 1     & 1     & 23    & 6   & 0   & 0 \\
    a1 b1 c1 c1 c1 c1 b2 & 440   & 6.37425 & 8.25986 & 1     & 1     & 23    & 6   & 0   & 0 \\
    a1 a1 B1 c1 b2 c1 c2 & 473   & 4.61158 & 8.41069 & 2     & 2     & 12    & 2, 6 & 0, 1 & 0, 1 \\
    a1 b1 b1 b1 b2 c1 c2 & 462   & 4.41948 & 8.50320 & 2     & 2     & 26    & 4, 4 & 0, 1 & 0, 1 \\
    a1 B1 C1 C1 B2 & 235   & 2.61803 & 8.51918 & 3     & 2     & 10    & 2, 4, 4 & 0, 2, 1 & 0, 1, 1 \\
    a1 b1 c1 b2 C2 C2 C2 & 431   & 3.99112 & 8.58353* & 6     & 2     & 25    & 1, 3, 3, 3, 3, 3 & 0, 5, 1, 2, 3, 4 & 0, 1, 1, 1, 1, 1 \\
    b1 c1 b1 b2 b2 c2 c2 & 131   & 4.61158 & 8.61242 & 1     & 1     & 10    & 6   & 0   & 0 \\
    b1 c1 c1 c1 b2 c2 b2 & 127   & 4.96069 & 8.71643 & 1     & 1     & 11    & 6   & 0   & 0 \\
    a1 a1 b1 c1 c1 c2 b2 & 480   & 3.25426 & 8.77866 & 2     & 2     & 12    & 4, 4 & 0, 1 & 0, 1 \\
    a1 b1 c1 B2 c1 c2 b2 & 429   & 3.13000 & 8.86560 & 2     & 2     & 11    & 4, 4 & 0, 1 & 0, 1 \\
		\hline
    \end{tabular}%
  \label{table:s2}%
\end{table}%
\end{landscape}

\begin{landscape}
\begin{table}[htbp]\small
  \centering
	\caption{Triangulations that SnapPy reports are non-geometric.}
    \begin{tabular}{|l|l|l|l|l|l|l|l|l|l|l|}
		\hline
    Dehn twist word & \specialcell{Isom.\\class} & \specialcell{Growth\\rate} & Volume & \#Sing. & \#Cusps & \#Tet. & Prongs & Sing. perm & \specialcell{Sing. to\\cusp}\\
		\hline
    a1 a1 a1 a1 b1 c1 b2 & 487   & 2.89005 & 4.85117 & 1     & 1     & 13    & 6   & 0   & 0 \\
    a1 a1 b2 c1 b1 b1 & 476   & 4.13016 & 7.16394 & 1     & 1     & 14    & 6   & 0   & 0 \\
    a1 a1 b1 b1 c1 c2 b2 & 482   & 4.21208 & 8.24198 & 2     & 2     & 26    & 4, 4 & 0, 1 & 0, 1 \\
    a1 a1 a1 b2 c1 b1 b1 & 483   & 6.11129 & 8.25157 & 1     & 1     & 23    & 6   & 0   & 0 \\
    a1 b1 c1 b2 C2 C2 C2 & 431   & 3.99112 & 8.58353 & 6     & 2     & 25    & 1, 3, 3, 3, 3, 3 & 0, 5, 1, 2, 3, 4 & 0, 1, 1, 1, 1, 1 \\
    a1 a1 b2 c1 b1 b1 b1 & 475   & 7.22040 & 9.14624 & 1     & 1     & 23    & 6   & 0   & 0 \\
    a1 b1 b1 b2 c1 c1 c2 & 451   & 5.52042 & 10.77524 & 2     & 2     & 26    & 4, 4 & 0, 1 & 0, 1 \\
    a1 C1 b2 c1 c1 B1 B1 & 189   & 9.16582 & 11.40320 & 1     & 1     & 21    & 6   & 0   & 0 \\
    a1 B1 a1 B1 C1 B2 C2 & 365   & 4.43066 & 12.86506 & 6     & 2     & 27    & 1, 3, 3, 3, 3, 3 & 0, 3, 1, 5, 2, 4 & 0, 1, 1, 1, 1, 1 \\
    b1 c1 B2 B2 B2 c2 B2 & 92    & 12.40010 & 13.37217 & 3     & 2     & 39    & 4, 3, 3 & 0, 2, 1 & 0, 1, 1 \\
    a1 B1 c1 B1 b2 c1 c2 & 355   & 7.99808 & 13.78589 & 3     & 2     & 21    & 2, 4, 4 & 0, 2, 1 & 0, 1, 1 \\
    a1 b1 c2 B2 c1 B2 c1 & 355   & 7.99808 & 13.78589 & 3     & 2     & 21    & 2, 4, 4 & 0, 2, 1 & 0, 1, 1 \\
    b1 c1 c1 B2 B2 c2 & 119   & 6.97984 & 14.17690 & 3     & 2     & 24    & 4, 3, 3 & 0, 2, 1 & 0, 1, 1 \\
    b1 b1 c1 c1 B2 B2 c2 & 119   & 7.87298 & 14.17690 & 2     & 2     & 24    & 4, 4 & 0, 1 & 0, 1 \\
    b1 B2 B2 c2 c2 B2 c1 & 10    & 14.48078 & 14.86313 & 3     & 2     & 31    & 4, 3, 3 & 0, 2, 1 & 0, 1, 1 \\
    b1 c1 B2 c2 c2 B2 B2 & 10    & 14.48078 & 14.86313 & 3     & 2     & 31    & 4, 3, 3 & 0, 2, 1 & 0, 1, 1 \\
    b1 c1 c1 B2 B2 B2 c2 & 117   & 9.89898 & 14.88705 & 3     & 2     & 33    & 4, 3, 3 & 0, 2, 1 & 0, 1, 1 \\
    b1 c1 B2 c1 B2 c2 & 100   & 8.19987 & 15.05432 & 3     & 2     & 31    & 4, 3, 3 & 0, 2, 1 & 0, 1, 1 \\
    b1 c1 B2 c1 c2 B2 & 100   & 8.19987 & 15.05432 & 3     & 2     & 31    & 4, 3, 3 & 0, 2, 1 & 0, 1, 1 \\
    b1 b1 c1 B2 c1 B2 c2 & 100   & 9.89898 & 15.05432 & 2     & 2     & 31    & 4, 4 & 0, 1 & 0, 1 \\
    b1 b1 c1 B2 c1 c2 B2 & 100   & 9.89898 & 15.05432 & 2     & 2     & 31    & 4, 4 & 0, 1 & 0, 1 \\
    b1 c1 c1 c1 B2 B2 c2 & 124   & 8.30401 & 15.60969 & 3     & 2     & 25    & 4, 3, 3 & 0, 2, 1 & 0, 1, 1 \\
    b1 c1 B2 B2 c1 B2 c2 & 95    & 12.21889 & 15.73326 & 3     & 2     & 47    & 4, 3, 3 & 0, 2, 1 & 0, 1, 1 \\
    a1 b1 C1 C1 b1 b2 C1 & 95    & 12.21889 & 15.73326 & 3     & 2     & 47    & 4, 3, 3 & 0, 2, 1 & 0, 1, 1 \\
    a1 B1 B1 C1 B1 B1 B2 & 297   & 4.19530 & 16.38405 & 5     & 3     & 28    & 2, 3, 3, 3, 3 & 0, 2, 1, 4, 3 & 0, 1, 1, 2, 2 \\
    b1 C1 b2 C1 b1 b2 c2 & 39    & 11.31474 & 16.62155 & 3     & 2     & 32    & 4, 3, 3 & 0, 2, 1 & 0, 1, 1 \\
		\hline
    \end{tabular}%
  \label{table:nongeo1}%
\end{table}%
\end{landscape}

\begin{landscape}
\begin{table}[htbp]\small
  \centering
	\caption{Triangulations that SnapPy reports are non-geometric. (Continuation of Table \ref{table:nongeo1})}
    \begin{tabular}{|l|l|l|l|l|l|l|l|l|l|l|}
		\hline
    Dehn twist word & \specialcell{Isom.\\class} & \specialcell{Growth\\rate} & Volume & \#Sing. & \#Cusps & \#Tet. & Prongs & Sing. perm & \specialcell{Sing. to\\cusp}\\
		\hline
    b1 c1 B2 c1 B2 c2 c2 & 39    & 11.31474 & 16.62155 & 3     & 2     & 32    & 4, 3, 3 & 0, 2, 1 & 0, 1, 1 \\
    b1 c1 B2 c1 B2 B2 c2 & 99    & 12.99411 & 16.80174 & 3     & 2     & 40    & 4, 3, 3 & 0, 2, 1 & 0, 1, 1 \\
    b1 c1 B2 c1 c2 B2 B2 & 99    & 12.99411 & 16.80174 & 3     & 2     & 40    & 4, 3, 3 & 0, 2, 1 & 0, 1, 1 \\
    a1 B1 c1 B2 c1 c1 C2 & 349   & 10.12331 & 17.14005 & 4     & 3     & 34    & 3, 3, 3, 3 & 0, 1, 3, 2 & 0, 1, 2, 2 \\
    a1 B1 c1 c1 B2 c1 C2 & 349   & 10.12331 & 17.14005 & 4     & 3     & 34    & 3, 3, 3, 3 & 0, 1, 3, 2 & 0, 1, 2, 2 \\
    b1 c1 B2 c1 c1 B2 c2 & 101   & 10.59023 & 17.31324 & 3     & 2     & 32    & 4, 3, 3 & 0, 2, 1 & 0, 1, 1 \\
    b1 c1 c1 B2 c1 c2 B2 & 101   & 10.59023 & 17.31324 & 3     & 2     & 32    & 4, 3, 3 & 0, 2, 1 & 0, 1, 1 \\
    b1 c1 B2 c2 B2 c1 B2 & 96    & 14.19162 & 17.64673 & 3     & 2     & 47    & 4, 3, 3 & 0, 2, 1 & 0, 1, 1 \\
    a1 b1 C1 C1 b2 C1 c2 & 399   & 8.90988 & 17.69647 & 5     & 3     & 32    & 2, 3, 3, 3, 3 & 0, 2, 1, 4, 3 & 0, 1, 1, 2, 2 \\
    a1 C1 C1 B1 B2 B2 C2 & 174   & 5.10315 & 17.96735 & 6     & 3     & 25    & 1, 3, 3, 3, 3, 3 & 0, 2, 1, 5, 3, 4 & 0, 1, 1, 2, 2, 2 \\
    a1 B1 C1 C1 C1 C1 B2 & 238   & 5.52042 & 18.05663 & 5     & 2     & 31    & 2, 3, 3, 3, 3 & 0, 2, 3, 4, 1 & 0, 1, 1, 1, 1 \\
    a1 b1 C1 b2 C1 c2 b2 & 407   & 7.84775 & 18.92441 & 5     & 3     & 33    & 2, 3, 3, 3, 3 & 0, 4, 3, 2, 1 & 0, 1, 2, 2, 1 \\
    a1 b1 C1 b2 C1 b2 c2 & 407   & 7.84775 & 18.92441 & 5     & 3     & 33    & 2, 3, 3, 3, 3 & 0, 4, 3, 2, 1 & 0, 1, 2, 2, 1 \\
    a1 B1 C2 B2 c1 B2 c1 & 195   & 8.73741 & 19.05808 & 4     & 3     & 35    & 3, 3, 3, 3 & 0, 1, 3, 2 & 0, 1, 2, 2 \\
    a1 B1 c1 B2 c1 B2 C2 & 195   & 8.73741 & 19.05808 & 4     & 3     & 35    & 3, 3, 3, 3 & 0, 1, 3, 2 & 0, 1, 2, 2 \\
    a1 B1 B1 C1 C1 C1 B2 & 292   & 5.96799 & 19.15756 & 5     & 2     & 38    & 2, 3, 3, 3, 3 & 0, 2, 4, 1, 3 & 0, 1, 1, 1, 1 \\
    a1 B1 C1 C1 C2 b2 b2 & 228   & 7.20200 & 19.51039 & 4     & 3     & 30    & 3, 3, 3, 3 & 0, 2, 1, 3 & 0, 1, 1, 2 \\
    a1 b1 b1 c1 B2 B2 c1 & 453   & 6.53277 & 20.27426 & 5     & 3     & 38    & 2, 3, 3, 3, 3 & 0, 2, 1, 4, 3 & 0, 1, 1, 2, 2 \\
    a1 B1 C1 C2 b2 C1 b2 & 217   & 8.51867 & 20.38781 & 4     & 3     & 37    & 3, 3, 3, 3 & 0, 2, 1, 3 & 0, 1, 1, 2 \\
    a1 B1 C1 b2 C1 C2 b2 & 217   & 8.51867 & 20.38781 & 4     & 3     & 37    & 3, 3, 3, 3 & 0, 2, 1, 3 & 0, 1, 1, 2 \\
    a1 B1 C1 C1 B1 B1 B2 & 244   & 8.02836 & 20.54293 & 5     & 3     & 45    & 2, 3, 3, 3, 3 & 0, 2, 1, 4, 3 & 0, 1, 1, 2, 2 \\
    a1 B1 B1 C1 C1 B1 B2 & 244   & 8.02836 & 20.54293 & 5     & 3     & 45    & 2, 3, 3, 3, 3 & 0, 2, 1, 4, 3 & 0, 1, 1, 2, 2 \\
		\hline
    \end{tabular}%
  \label{table:nongeo2}%
\end{table}%
\end{landscape}

\nocite{SnapPea}
\phantomsection
\bibliography{references}{}

\begin{thebibliography}{10}

\bibitem{MR2866919}
Ian Agol.
\newblock Ideal triangulations of pseudo-{A}nosov mapping tori.
\newblock In {\em Topology and geometry in dimension three}, volume 560 of {\em
  Contemp. Math.}, pages 1--17. Amer. Math. Soc., Providence, RI, 2011.

\bibitem{Twister}
Mark Bell, Tracy Hall, and Saul Schleimer.
\newblock \emph{Twister} [computer software].
\newblock \url{http://surfacebundles.wordpress.com/}, 2012.
\newblock Included with SnapPy version 1.5.

\bibitem{BH}
M.~Bestvina and M.~Handel.
\newblock Train-tracks for surface homeomorphisms.
\newblock {\em Topology}, 34(1):109--140, 1995.

\bibitem{regina}
Benjamin~A. Burton.
\newblock Regina: Normal surface and 3-manifold topology software.
\newblock {http://\allowbreak regina.\allowbreak sourceforge.\allowbreak net/},
  1999--2013.

\bibitem{MR2327361}
Danny Calegari.
\newblock {\em Foliations and the geometry of 3-manifolds}.
\newblock Oxford Mathematical Monographs. Oxford University Press, Oxford,
  2007.

\bibitem{CBNotes}
A.~Casson and S.~Bleiler.
\newblock Automorphisms of surfaces after {N}ielsen and {T}hurston.
\newblock Handwritten notes, University of Texas at Austin, Volume 1 Fall 1982,
  Volume 2 Spring 1983.

\bibitem{CAS}
Andrew~J. Casson and Steven~A. Bleiler.
\newblock {\em Automorphisms of surfaces after {N}ielsen and {T}hurston},
  volume~9 of {\em London Mathematical Society Student Texts}.
\newblock Cambridge University Press, Cambridge, 1988.

\bibitem{MR1758805}
David Coulson, Oliver~A. Goodman, Craig~D. Hodgson, and Walter~D. Neumann.
\newblock Computing arithmetic invariants of 3-manifolds.
\newblock {\em Experiment. Math.}, 9(1):127--152, 2000.

\bibitem{SnapPy}
Marc Culler and Nathan Dunfield.
\newblock \emph{Snap{P}y} (version 1.5) [computer software].
\newblock \url{http://www.math.uic.edu/t3m/SnapPy/}, 2012.
\newblock Uses the SnapPea kernel written by Jeff Weeks.

\bibitem{MR2850125}
Benson Farb and Dan Margalit.
\newblock {\em A primer on mapping class groups}, volume~49 of {\em Princeton
  Mathematical Series}.
\newblock Princeton University Press, Princeton, NJ, 2012.

\bibitem{FLP}
A.~Fathi, F.~Laudenbach, V.~Po\'enaru, and et~al.
\newblock {\em {T}hurston's {W}ork on {S}urfaces}.
\newblock Princeton University Press, 2012.
\newblock Translated from the French original by Djun Kim and Dan Margalit.

\bibitem{Stefano}
Stefano Francaviglia.
\newblock Hyperbolicity equations for cusped 3-manifolds and volume-rigidity of
  representations.
\newblock Ph.D. thesis, Scuola Normale Superiore, 2004.

\bibitem{FGAng}
David Futer and Fran{\c{c}}ois Gu{\'e}ritaud.
\newblock Explicit angle structures for veering triangulations.
\newblock {\em Algebr. Geom. Topol.}, 13(1):205--235, 2013.

\bibitem{Snap}
Oliver Goodman.
\newblock \emph{Snap} (version 1.5) [computer software].
\newblock \url{http://www.ms.unimelb.edu.au/~snap/}.
\newblock Uses the SnapPea kernel written by Jeff Weeks.

\bibitem{MR2255497}
Fran{\c{c}}ois Gu{\'e}ritaud.
\newblock On canonical triangulations of once-punctured torus bundles and
  two-bridge link complements.
\newblock {\em Geom. Topol.}, 10:1239--1284, 2006.
\newblock With an appendix by David Futer.

\bibitem{Trains}
Toby Hall.
\newblock \emph{Trains} (version 4.32, 2009) [computer software].
\newblock \url{http://www.liv.ac.uk/~tobyhall/T_Hall.html}.
\newblock An implementation of the Bestvina-Handel algorithm.

\bibitem{MR2860987}
Craig~D. Hodgson, J.~Hyam Rubinstein, Henry Segerman, and Stephan Tillmann.
\newblock Veering triangulations admit strict angle structures.
\newblock {\em Geom. Topol.}, 15(4):2073--2089, 2011.

\bibitem{Veering}
Ahmad Issa.
\newblock \emph{Veering} ({J}uly 2013) [computer software].
\newblock \url{http://www.ms.unimelb.edu.au/~veering/}.
\newblock Uses the Trains program written by Toby Hall.

\bibitem{Issa}
Ahmad Issa.
\newblock Construction of non-geometric veering triangulations of fibered
  hyperbolic 3-manifolds.
\newblock MSc thesis, University of Melbourne, 2012.

\bibitem{SnapPea}
Weeks Jeffrey.
\newblock \emph{SnapPea} [computer software].
\newblock \url{http://www.geometrygames.org/SnapPea/}.

\bibitem{MR1988201}
Marc Lackenby.
\newblock The canonical decomposition of once-punctured torus bundles.
\newblock {\em Comment. Math. Helv.}, 78(2):363--384, 2003.

\bibitem{MR2179010}
William Menasco and Morwen Thistlethwaite, editors.
\newblock {\em Handbook of knot theory}.
\newblock Elsevier B. V., Amsterdam, 2005.

\bibitem{Mosher2}
Lee Mosher.
\newblock Train track expansions of measured foliations.
\newblock preprint, December, 2003.

\bibitem{MR1663915}
Walter~D. Neumann and Jun Yang.
\newblock Bloch invariants of hyperbolic {$3$}-manifolds.
\newblock {\em Duke Math. J.}, 96(1):29--59, 1999.

\bibitem{PENNER}
R.~C. Penner and J.~L. Harer.
\newblock {\em Combinatorics of train tracks}, volume 125 of {\em Annals of
  Mathematics Studies}.
\newblock Princeton University Press, Princeton, NJ, 1992.

\bibitem{GenDef}
Henry Segerman.
\newblock A generalisation of the deformation variety.
\newblock {\em Algebr. Geom. Topol.}, 12(4):2179--2244, 2012.

\bibitem{ThNotes}
W.~P. Thurston.
\newblock The geometry and topology of 3-manifolds.
\newblock Mimeographed notes, Princeton University Mathematics Department,
  1979.

\end{thebibliography}
\bibliographystyle{plain}

\end{document}